\documentclass[11pt]{amsart}
\usepackage[utf8]{inputenc}
\usepackage[english]{babel}
\usepackage{graphicx}
\usepackage{amssymb}
\usepackage{amsmath}

\usepackage{tikz}
\usetikzlibrary{decorations.markings}

\theoremstyle{plain}	
\newtheorem{theorem}{Theorem}[section]
\newtheorem{proposition}[theorem]{Proposition}

\theoremstyle{definition} 
\newtheorem{definition}[theorem]{Definition}
\newtheorem{example}[theorem]{Example}

\title{On 2-dimensional Dijkgraaf-Witten Theory with Defects}
\author{Aria L. Dougherty}
\address[Aria L. Dougherty]{Mary Baldwin College, Staunton, VA 24401 }
\email{doughertyal4843@mbc.edu}

\author{Hwajin Park}
\address[Hwajin Park]{University of Virginia, Charlottesville, VA 22904 }
\email{hp3aq@virginia.edu}
 
\author{David N. Yetter}
\address[David N. Yetter]{Kansas State University, Manhattan, KS 66506}
\email[Corresponding author]{dyetter@math.ksu.edu}

\begin{document}
\tikzset{midarrow/.style={
		decoration={markings,
			mark= at position 0.5 with {\arrow{#1}},
		},
		postaction={decorate}
	}
}

\begin{abstract}In this paper, we provide a construction of a state-sum model for finite gauge-group Dijkgraaf-Witten theory on surfaces with codimension 1 defects.   The construction requires not only that the triangulation be subordinate to the filtration, but flag-like:  each simplex of the triangulation is either disjoint from the defect curve, or intersects it in a closed face.  The construction allows internal degrees of freedom in the defect curves by introducing a second gauge-group from which edges of the curve are labeled in the state-sum construction.   Edges incident with the defect, but not lying in it, have states lying in a set with commuting actions of the two gauge-groups.  We determine the appropriate generalizations of the 2-cocycles specifying twistings of defect-free 2D Dijkgraaf-Witten theory.  Examples arising by restriction of group 2- cocycles, and constructed from characters of the 2-dimensional guage group are presented.  This research was carried out at Summer Undergraduate Mathematics Research (SUMaR) math REU at Kansas State University, funded by NSF under DMS award \#1262877.
\end{abstract}

\maketitle

\section{Introduction}
One standard construction of topological quantum field theories is given by an analogue of lattice-gauge theories:  state-sum constructions on triangulated manifolds (cf., for instance \cite{TV}, \cite{W}, \cite{Fr}) in which simplexes of various dimension are labeled or ``colored'' with combinatorial or algebraic data, sometimes called``local states'', subject to conditions traditionally termed ``admissibility'', values are computed from the local states and multiplied over the whole triangulation, and these products in turn are summed over all admissible labelings -- or alternatively over all labelings with the local state value of inadmissible colorings set to zero. 
The resulting state-sum is then normalized so as to be invariant under changes to the triangulation within the same piecewise linear (PL) structure on the space, giving a topological invariant, which by standard constructions (cf. \cite{Y}) is seen to arise from a topological quantum field theory in the sense of Atiyah \cite{Atiyah}.

In $n$-dimensional Dijkgraaf-Witten theory, the $n$-simplex edges are labeled with elements from a finite group, and the state values are defined by an $n$-dimensional group cocycle \cite{W}, \cite{Fr}.  It has been argued by Fuchs et al. \cite{Fu} at the physical level of rigor that the extension of a TQFT by the inclusion of defects -- submanifolds of lower dimension across which the bulk structure of the TQFT does not extend -- will naturally result in data on defects having an action of the bulk data.  In the case of Dikjgraaf-Witten theory, this suggests that defects will have data corresponding to a set with an action of the bulk gauge group.  In this paper we introduce a curve $C$ to the surface $\Sigma$ and find an invariant of $\Sigma\supseteq C$ by way of the combinatorial description of Dijkgraaf-Witten theory.

Throughout this paper, all curves and surfaces are oriented, compact, without boundary, and, for simplicity, equipped with a specified PL structure.

\section{Triangulations of surfaces with curves}

It turns out that merely requiring the curve be a subcomplex of the triangulated surface is insufficient.  We require that triangulations of surfaces with curves be flag-like:

\begin{definition}
A triangulation $\mathcal{T}$ of a surface with curve $\Sigma\supseteq C$ is \textit{flag-like} if $C$ is a subcomplex, and if, for every 2-simplex $\sigma\in\mathcal{T}$, the intersection of $\sigma$ and $C$ is either a face (i.e. one vertex or one edge) or is empty.
\end{definition}

Note that the definition given is specific to a surface containing a curve.  The notion is more general and applies to filtered piecewise linear spaces, in which case the condition for a triangulation to be flag-like is that the induced filtration of each simplex be a (possibly incomplete) simplicial flag (cf. literature on intersection (co)homology, where the notion first arose, as, for example \cite{Friedman}) hence the name.

The restriction to flag-like triangulations is very mild, since it is easy to see that the barycentric subdivision of any triangulation for which the subspaces of a filtration are subcomplexes is flag-like.

In all figures throughout the rest of the paper, shaded edges and vertices indicate those lying in the curve.

Three distinct types of 2-simplexes may occur in flag-like triangulations of a surface with curve:
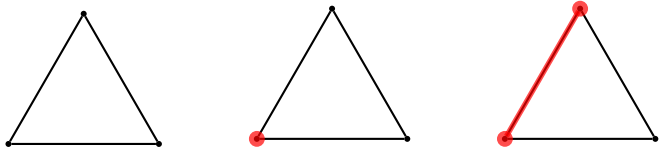
\begin{align*}
\setlength{\tabcolsep}{1.5em}{
\begin{tabular}{c c c}
	\begin{tikzpicture}[thick] 
		\node[circle, fill, inner sep=.8pt, outer sep=0pt] (A) at (0,0){};
		\node[circle, fill, inner sep=.8pt, outer sep=0pt] (B) at (1,1.73205){};
		\node[circle, fill, inner sep=.8pt, outer sep=0pt] (C) at (2,0){};
		\draw (A)--(B)--(C)--(A);
	\end{tikzpicture}
	&
	\begin{tikzpicture}[thick]
		\node[circle, fill, inner sep=.8pt, outer sep=0pt] (A) at (0,0){};
		\node[circle, fill, inner sep=.8pt, outer sep=0pt] (B) at (1,1.73205){};
		\node[circle, fill, inner sep=.8pt, outer sep=0pt] (C) at (2,0){};
		\draw (A)--(B)--(C)--(A);
		\fill [red, opacity=.7] (A) circle (3pt);
	\end{tikzpicture}
	&
	\begin{tikzpicture}[thick]
		\node[circle, fill, inner sep=.8pt, outer sep=0pt] (A) at (0,0){};
		\node[circle, fill, inner sep=.8pt, outer sep=0pt] (B) at (1,1.73205){};
		\node[circle, fill, inner sep=.8pt, outer sep=0pt] (C) at (2,0){};
		\draw (A)--(B)--(C)--(A);
		\draw[line width=2pt, red, opacity=.7] (A)--(B);
		\fill [red, opacity=.7] (A) circle (3pt);
		\fill [red, opacity=.7] (B) circle (3pt);
	\end{tikzpicture}
\end{tabular}
}
\end{align*}

Alexander, in his seminal paper \cite{A}, showed that two PL spaces are PL homeomorphic if and only if they admit iterated stellar subdivisions with the same combinatorial type.  Usually this result is stated in a form which does not require the reaching of a common subdivision, but allows a sequences of stellar subdivisions and their inverses (welds), both of which are now commonly called ``Alexander moves,'' but we will have cause to use the full force of the result.  Thus topological invariance of some quantity or object computed or constructed from a triangulation may be shown by its invariance under the Alexander moves.

The Alexander moves applicable to surfaces, are depicted below.

$$
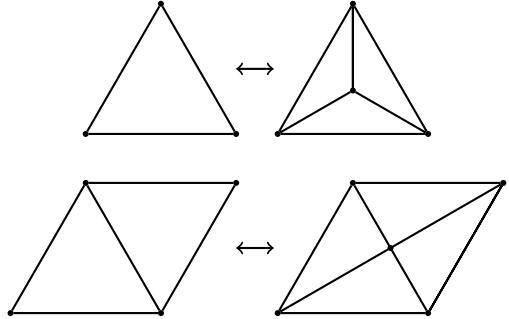
\begin{tikzpicture}[thick] 
	\node[circle, fill, inner sep=.8pt, outer sep=0pt] (A) at (0,0){};
	\node[circle, fill, inner sep=.8pt, outer sep=0pt] (B) at (1,1.73205){};
	\node[circle, fill, inner sep=.8pt, outer sep=0pt] (C) at (2,0){};
	\draw (A)--(B)--(C)--(A);		
	\draw [<->] (2,.866)--(2.5,.866);
\end{tikzpicture}
\begin{tikzpicture}[thick] 
	\node[circle, fill, inner sep=.8pt, outer sep=0pt] (A) at (0,0){};
	\node[circle, fill, inner sep=.8pt, outer sep=0pt] (B) at (1,1.73205){};
	\node[circle, fill, inner sep=.8pt, outer sep=0pt] (C) at (2,0){};
	\node[circle, fill, inner sep=.8pt, outer sep=0pt] (D) at (1,.57735){};
	\draw (A)--(B)--(D)--(A)--(C)--(D)--(B)--(C);
\end{tikzpicture}$$

$$\begin{tikzpicture}[thick] 
	\node[circle, fill, inner sep=.8pt, outer sep=0pt] (A) at (0,0){};
	\node[circle, fill, inner sep=.8pt, outer sep=0pt] (B) at (1,1.73205){};
	\node[circle, fill, inner sep=.8pt, outer sep=0pt] (C) at (2,0){};
	\node[circle, fill, inner sep=.8pt, outer sep=0pt] (D) at (3,1.73205){};
	\draw (B)--(D)--(C)--(A)--(B)--(C);
	\draw [<->] (3,.866)--(3.5,.866);
\end{tikzpicture}
\begin{tikzpicture}[thick] 
	\node[circle, fill, inner sep=.8pt, outer sep=0pt] (A) at (0,0){};
	\node[circle, fill, inner sep=.8pt, outer sep=0pt] (B) at (1,1.73205){};
	\node[circle, fill, inner sep=.8pt, outer sep=0pt] (C) at (2,0){};
	\node[circle, fill, inner sep=.8pt, outer sep=0pt] (D) at (3,1.73205){};
	\node[circle, fill, inner sep=.8pt, outer sep=0pt] (E) at (1.5,.866){};
	\draw (B)--(D)--(C)--(A)--(B)--(C)--(D)--(A);
\end{tikzpicture}$$

Once a curve is introduced, it is easy to see that if an Alexander subdivision is applied to a flag-like triangluation, the resulting triangulation will be flag-like.  This is not the case with welds, as for example
in the two depicted below

$$
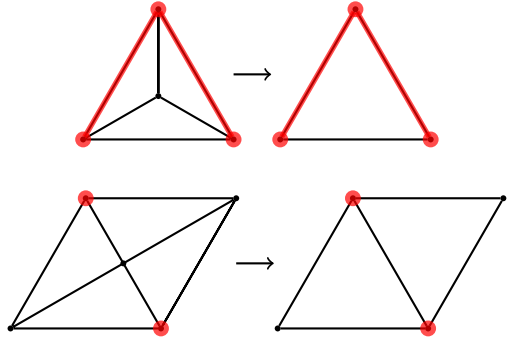
\begin{tikzpicture}[thick] 
	\node[circle, fill, inner sep=.8pt, outer sep=0pt] (A) at (0,0){};
	\node[circle, fill, inner sep=.8pt, outer sep=0pt] (B) at (1,1.73205){};
	\node[circle, fill, inner sep=.8pt, outer sep=0pt] (C) at (2,0){};
	\node[circle, fill, inner sep=.8pt, outer sep=0pt] (D) at (1,.57735){};
	\draw (A)--(B)--(D)--(A)--(C)--(D)--(B)--(C);
                  \draw[line width=2pt, red, opacity=.7] (A)--(B);
                  \draw[line width=2pt, red, opacity=.7] (B)--(C);
		\fill [red, opacity=.7] (A) circle (3pt);
		\fill [red, opacity=.7] (B) circle (3pt);
		\fill [red, opacity=.7] (C) circle (3pt);
         \draw [->] (2,.866)--(2.5,.866);
\end{tikzpicture}
\begin{tikzpicture}[thick] 
	\node[circle, fill, inner sep=.8pt, outer sep=0pt] (A) at (0,0){};
	\node[circle, fill, inner sep=.8pt, outer sep=0pt] (B) at (1,1.73205){};
	\node[circle, fill, inner sep=.8pt, outer sep=0pt] (C) at (2,0){};
	\draw (A)--(B)--(C)--(A);
                  \draw[line width=2pt, red, opacity=.7] (A)--(B);
                  \draw[line width=2pt, red, opacity=.7] (B)--(C);
		\fill [red, opacity=.7] (A) circle (3pt);
		\fill [red, opacity=.7] (B) circle (3pt);
		\fill [red, opacity=.7] (C) circle (3pt);
\end{tikzpicture}
$$

$$\begin{tikzpicture}[thick] 
	\node[circle, fill, inner sep=.8pt, outer sep=0pt] (A) at (0,0){};
	\node[circle, fill, inner sep=.8pt, outer sep=0pt] (B) at (1,1.73205){};
	\node[circle, fill, inner sep=.8pt, outer sep=0pt] (C) at (2,0){};
	\node[circle, fill, inner sep=.8pt, outer sep=0pt] (D) at (3,1.73205){};
	\node[circle, fill, inner sep=.8pt, outer sep=0pt] (E) at (1.5,.866){};
	\draw (B)--(D)--(C)--(A)--(B)--(C)--(D)--(A);
		\fill [red, opacity=.7] (B) circle (3pt);
		\fill [red, opacity=.7] (C) circle (3pt);
	\draw [->] (3,.866)--(3.5,.866);
\end{tikzpicture}
\begin{tikzpicture}[thick] 
	\node[circle, fill, inner sep=.8pt, outer sep=0pt] (A) at (0,0){};
	\node[circle, fill, inner sep=.8pt, outer sep=0pt] (B) at (1,1.73205){};
	\node[circle, fill, inner sep=.8pt, outer sep=0pt] (C) at (2,0){};
	\node[circle, fill, inner sep=.8pt, outer sep=0pt] (D) at (3,1.73205){};
	\draw (B)--(D)--(C)--(A)--(B)--(C);
		\fill [red, opacity=.7] (B) circle (3pt);
		\fill [red, opacity=.7] (C) circle (3pt);
\end{tikzpicture}
$$

We distinguish moves which preserve flag-likeness by making

\begin{definition}
A combinatorial move applicable to a flag-like triangulation is {\em flag-like} if its final state is also a flag-like triangulation.
\end{definition}

We then have

\begin{theorem} \label{Alexander_moves_suffice}
If $(\Sigma_i \supseteq C_i, {\mathcal T}_i)$ for $i = 1, 2$ are two surface-curve pairs equipped with flag-like triangulations, there is a PL-homeomorphism from $\Sigma_1$ to $\Sigma_2$ carrying $C_1$ to $C_2$ if and only if there is sequence of flag-like Alexander moves applicable to ${\mathcal T}_1$ which results in a triangulation combinatorially isomorphic to ${\mathcal T}_2$.
\end{theorem} 

\begin{proof} The forward implication follows from the full force of Alexander's theorem -- that two complexes are PL-homeomorphic if and only if there is a common subdivision by Alexander subdivisions --  and the observation that subdivisions of flag-like triangulations are always flag-like:  starting at each of
${\mathcal T}_i$, $i = 1,2$, apply Alexander subdivisions (necessarily flag-like) until a common combinatorial type of triangulation is reached.  The sequence in the theorem is then the sequence of subdivisions beginning at ${\mathcal T}_1$ followed by the welds reversing the subdivisions starting at
${\mathcal T}_2$ in reverse order.  The converse is trivial.
\end{proof}

Exhaustive consideration of cases shows that there are ten combinatorial types of flag-like Alexander subdivision moves (and thus ten corresponding types of flag-like Alexander welds).  They are depicted below:

$$
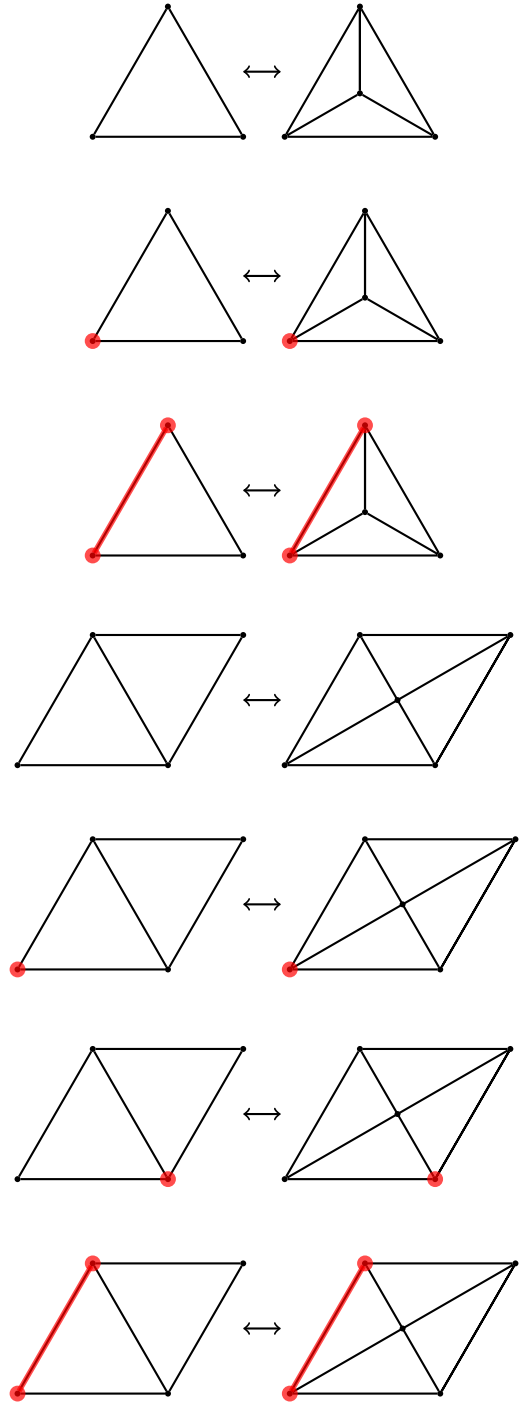
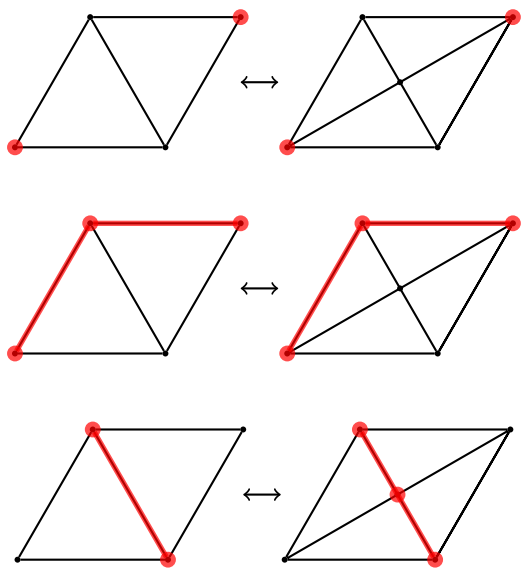
\begin{tikzpicture}[thick] 
	\node[circle, fill, inner sep=.8pt, outer sep=0pt] (A) at (0,0){};
	\node[circle, fill, inner sep=.8pt, outer sep=0pt] (B) at (1,1.73205){};
	\node[circle, fill, inner sep=.8pt, outer sep=0pt] (C) at (2,0){};
	\draw (A)--(B)--(C)--(A);
	\draw [<->] (2,.866)--(2.5,.866);
\end{tikzpicture}
\begin{tikzpicture}[thick] 
	\node[circle, fill, inner sep=.8pt, outer sep=0pt] (A) at (0,0){};
	\node[circle, fill, inner sep=.8pt, outer sep=0pt] (B) at (1,1.73205){};
	\node[circle, fill, inner sep=.8pt, outer sep=0pt] (C) at (2,0){};
	\node[circle, fill, inner sep=.8pt, outer sep=0pt] (D) at (1,.57735){};
	\draw (A)--(B)--(D)--(A)--(C)--(D)--(B)--(C);
\end{tikzpicture}$$

$$\begin{tikzpicture}[thick] 
	\node[circle, fill, inner sep=.8pt, outer sep=0pt] (A) at (0,0){};
	\node[circle, fill, inner sep=.8pt, outer sep=0pt] (B) at (1,1.73205){};
	\node[circle, fill, inner sep=.8pt, outer sep=0pt] (C) at (2,0){};
	\draw (A)--(B)--(C)--(A);
	\fill[red, opacity=.7] (A) circle (3pt);
	\draw [<->] (2,.866)--(2.5,.866);
\end{tikzpicture}
\begin{tikzpicture}[thick] 
	\node[circle, fill, inner sep=.8pt, outer sep=0pt] (A) at (0,0){};
	\node[circle, fill, inner sep=.8pt, outer sep=0pt] (B) at (1,1.73205){};
	\node[circle, fill, inner sep=.8pt, outer sep=0pt] (C) at (2,0){};	
	\node[circle, fill, inner sep=.8pt, outer sep=0pt] (D) at (1,.57735){};
	\draw (A)--(B)--(D)--(A)--(C)--(D)--(B)--(C);
	\fill[red, opacity=.7] (A) circle (3pt);
\end{tikzpicture}$$

$$\begin{tikzpicture}[thick] 
	\node[circle, fill, inner sep=.8pt, outer sep=0pt] (A) at (0,0){};
	\node[circle, fill, inner sep=.8pt, outer sep=0pt] (B) at (1,1.73205){};
	\node[circle, fill, inner sep=.8pt, outer sep=0pt] (C) at (2,0){};
	\draw (A)--(B)--(C)--(A);
	\draw[line width=2pt, red, opacity=.7] (A)--(B);
	\fill[red, opacity=.7] (A) circle (3pt);
	\fill[red, opacity=.7] (B) circle (3pt);
	\draw [<->] (2,.866)--(2.5,.866);
\end{tikzpicture}
\begin{tikzpicture}[thick] 
	\node[circle, fill, inner sep=.8pt, outer sep=0pt] (A) at (0,0){};
	\node[circle, fill, inner sep=.8pt, outer sep=0pt] (B) at (1,1.73205){};
	\node[circle, fill, inner sep=.8pt, outer sep=0pt] (C) at (2,0){};
	\node[circle, fill, inner sep=.8pt, outer sep=0pt] (D) at (1,.57735){};
	\draw (A)--(B)--(D)--(A)--(C)--(D)--(B)--(C);
	\draw[line width=2pt, red, opacity=.7] (A)--(B);
	\fill[red, opacity=.7] (A) circle (3pt);
	\fill[red, opacity=.7] (B) circle (3pt);
\end{tikzpicture}$$

$$\begin{tikzpicture}[thick] 
	\node[circle, fill, inner sep=.8pt, outer sep=0pt] (A) at (0,0){};
	\node[circle, fill, inner sep=.8pt, outer sep=0pt] (B) at (1,1.73205){};
	\node[circle, fill, inner sep=.8pt, outer sep=0pt] (C) at (2,0){};
	\node[circle, fill, inner sep=.8pt, outer sep=0pt] (D) at (3,1.73205){};
	\draw (B)--(D)--(C)--(A)--(B)--(C);
	\draw [<->] (3,.866)--(3.5,.866);
\end{tikzpicture}
\begin{tikzpicture}[thick] 
	\node[circle, fill, inner sep=.8pt, outer sep=0pt] (A) at (0,0){};
	\node[circle, fill, inner sep=.8pt, outer sep=0pt] (B) at (1,1.73205){};
	\node[circle, fill, inner sep=.8pt, outer sep=0pt] (C) at (2,0){};
	\node[circle, fill, inner sep=.8pt, outer sep=0pt] (D) at (3,1.73205){};
	\node[circle, fill, inner sep=.8pt, outer sep=0pt] (E) at (1.5,.866){};
	\draw (B)--(D)--(C)--(A)--(B)--(C)--(D)--(A);
\end{tikzpicture}$$

$$\begin{tikzpicture}[thick] 
	\node[circle, fill, inner sep=.8pt, outer sep=0pt] (A) at (0,0){};
	\node[circle, fill, inner sep=.8pt, outer sep=0pt] (B) at (1,1.73205){};
	\node[circle, fill, inner sep=.8pt, outer sep=0pt] (C) at (2,0){};
	\node[circle, fill, inner sep=.8pt, outer sep=0pt] (D) at (3,1.73205){};
	\draw (B)--(D)--(C)--(A)--(B)--(C);
	\fill[red, opacity=.7] (A) circle (3pt);
		\draw [<->] (3,.866)--(3.5,.866);
\end{tikzpicture}
\begin{tikzpicture}[thick] 
	\node[circle, fill, inner sep=.8pt, outer sep=0pt] (A) at (0,0){};
	\node[circle, fill, inner sep=.8pt, outer sep=0pt] (B) at (1,1.73205){};
	\node[circle, fill, inner sep=.8pt, outer sep=0pt] (C) at (2,0){};
	\node[circle, fill, inner sep=.8pt, outer sep=0pt] (D) at (3,1.73205){};
	\node[circle, fill, inner sep=.8pt, outer sep=0pt] (E) at (1.5,.866){};
	\draw (B)--(D)--(C)--(A)--(B)--(C)--(D)--(A);
	\fill[red, opacity=.7] (A) circle (3pt);
\end{tikzpicture}$$

$$\begin{tikzpicture}[thick] 
	\node[circle, fill, inner sep=.8pt, outer sep=0pt] (A) at (0,0){};
	\node[circle, fill, inner sep=.8pt, outer sep=0pt] (B) at (1,1.73205){};
	\node[circle, fill, inner sep=.8pt, outer sep=0pt] (C) at (2,0){};	
	\node[circle, fill, inner sep=.8pt, outer sep=0pt] (D) at (3,1.73205){};
	\draw (B)--(D)--(C)--(A)--(B)--(C);
	\draw [<->] (3,.866)--(3.5,.866);
	\fill[red, opacity=.7] (C) circle (3pt);
\end{tikzpicture}
\begin{tikzpicture}[thick] 
	\node[circle, fill, inner sep=.8pt, outer sep=0pt] (A) at (0,0){};
	\node[circle, fill, inner sep=.8pt, outer sep=0pt] (B) at (1,1.73205){};
	\node[circle, fill, inner sep=.8pt, outer sep=0pt] (C) at (2,0){};
	\node[circle, fill, inner sep=.8pt, outer sep=0pt] (D) at (3,1.73205){};
	\node[circle, fill, inner sep=.8pt, outer sep=0pt] (E) at (1.5,.866){};
	\draw (B)--(D)--(C)--(A)--(B)--(C)--(D)--(A);
	\fill[red, opacity=.7] (C) circle (3pt);
\end{tikzpicture}$$

$$\begin{tikzpicture}[thick] 
	\node[circle, fill, inner sep=.8pt, outer sep=0pt] (A) at (0,0){};
	\node[circle, fill, inner sep=.8pt, outer sep=0pt] (B) at (1,1.73205){};
	\node[circle, fill, inner sep=.8pt, outer sep=0pt] (C) at (2,0){};
	\node[circle, fill, inner sep=.8pt, outer sep=0pt] (D) at (3,1.73205){};
	\draw (B)--(D)--(C)--(A)--(B)--(C);
	\draw [<->] (3,.866)--(3.5,.866);
	\fill[red, opacity=.7] (A) circle (3pt);
	\fill[red, opacity=.7] (B) circle (3pt);
	\draw[line width=2pt, red, opacity=.7] (A)--(B);
\end{tikzpicture}
\begin{tikzpicture}[thick] 
	\node[circle, fill, inner sep=.8pt, outer sep=0pt] (A) at (0,0){};
	\node[circle, fill, inner sep=.8pt, outer sep=0pt] (B) at (1,1.73205){};
	\node[circle, fill, inner sep=.8pt, outer sep=0pt] (C) at (2,0){};
	\node[circle, fill, inner sep=.8pt, outer sep=0pt] (D) at (3,1.73205){};
	\node[circle, fill, inner sep=.8pt, outer sep=0pt] (E) at (1.5,.866){};
	\draw (B)--(D)--(C)--(A)--(B)--(C)--(D)--(A);
	\fill[red, opacity=.7] (A) circle (3pt);
	\fill[red, opacity=.7] (B) circle (3pt);
	\draw[line width=2pt, red, opacity=.7] (A)--(B);
\end{tikzpicture}$$

$$\begin{tikzpicture}[thick] 
	\node[circle, fill, inner sep=.8pt, outer sep=0pt] (A) at (0,0){};
	\node[circle, fill, inner sep=.8pt, outer sep=0pt] (B) at (1,1.73205){};
	\node[circle, fill, inner sep=.8pt, outer sep=0pt] (C) at (2,0){};
	\node[circle, fill, inner sep=.8pt, outer sep=0pt] (D) at (3,1.73205){};
	\draw (B)--(D)--(C)--(A)--(B)--(C);
	\draw [<->] (3,.866)--(3.5,.866);
	\fill[red, opacity=.7] (A) circle (3pt);
	\fill[red, opacity=.7] (D) circle (3pt);
\end{tikzpicture}
\begin{tikzpicture}[thick] 
	\node[circle, fill, inner sep=.8pt, outer sep=0pt] (A) at (0,0){};
	\node[circle, fill, inner sep=.8pt, outer sep=0pt] (B) at (1,1.73205){};
	\node[circle, fill, inner sep=.8pt, outer sep=0pt] (C) at (2,0){};
	\node[circle, fill, inner sep=.8pt, outer sep=0pt] (D) at (3,1.73205){};
	\node[circle, fill, inner sep=.8pt, outer sep=0pt] (E) at (1.5,.866){};
	\draw (B)--(D)--(C)--(A)--(B)--(C)--(D)--(A);
	\fill[red, opacity=.7] (A) circle (3pt);
	\fill[red, opacity=.7] (D) circle (3pt);
\end{tikzpicture}$$

$$\begin{tikzpicture}[thick] 
	\node[circle, fill, inner sep=.8pt, outer sep=0pt] (A) at (0,0){};
	\node[circle, fill, inner sep=.8pt, outer sep=0pt] (B) at (1,1.73205){};
	\node[circle, fill, inner sep=.8pt, outer sep=0pt] (C) at (2,0){};
	\node[circle, fill, inner sep=.8pt, outer sep=0pt] (D) at (3,1.73205){};
	\draw (B)--(D)--(C)--(A)--(B)--(C);
	\draw [<->] (3,.866)--(3.5,.866);
	\fill[red, opacity=.7] (A) circle (3pt);
	\fill[red, opacity=.7] (B) circle (3pt);
	\fill[red, opacity=.7] (D) circle (3pt);
	\draw[line width=2pt, red, opacity=.7] (A)--(B)--(D);
\end{tikzpicture}
\begin{tikzpicture}[thick] 
	\node[circle, fill, inner sep=.8pt, outer sep=0pt] (A) at (0,0){};
	\node[circle, fill, inner sep=.8pt, outer sep=0pt] (B) at (1,1.73205){};
	\node[circle, fill, inner sep=.8pt, outer sep=0pt] (C) at (2,0){};
	\node[circle, fill, inner sep=.8pt, outer sep=0pt] (D) at (3,1.73205){};
	\node[circle, fill, inner sep=.8pt, outer sep=0pt] (E) at (1.5,.866){};
	\draw (B)--(D)--(C)--(A)--(B)--(C)--(D)--(A);
	\fill[red, opacity=.7] (A) circle (3pt);
	\fill[red, opacity=.7] (B) circle (3pt);
	\fill[red, opacity=.7] (D) circle (3pt);
	\draw[line width=2pt, red, opacity=.7] (A)--(B)--(D);
\end{tikzpicture}$$

$$\begin{tikzpicture}[thick] 
	\node[circle, fill, inner sep=.8pt, outer sep=0pt] (A) at (0,0){};
	\node[circle, fill, inner sep=.8pt, outer sep=0pt] (B) at (1,1.73205){};
	\node[circle, fill, inner sep=.8pt, outer sep=0pt] (C) at (2,0){};
	\node[circle, fill, inner sep=.8pt, outer sep=0pt] (D) at (3,1.73205){};
	\draw (B)--(D)--(C)--(A)--(B)--(C);
	\draw [<->] (3,.866)--(3.5,.866);
	\fill[red, opacity=.7] (C) circle (3pt);
	\fill[red, opacity=.7] (B) circle (3pt);
	\draw[line width=2pt, red, opacity=.7] (C)--(B);
\end{tikzpicture}
\begin{tikzpicture}[thick] 
	\node[circle, fill, inner sep=.8pt, outer sep=0pt] (A) at (0,0){};
	\node[circle, fill, inner sep=.8pt, outer sep=0pt] (B) at (1,1.73205){};
	\node[circle, fill, inner sep=.8pt, outer sep=0pt] (C) at (2,0){};
	\node[circle, fill, inner sep=.8pt, outer sep=0pt] (D) at (3,1.73205){};
	\node[circle, fill, inner sep=.8pt, outer sep=0pt] (E) at (1.5,.866){};
	\draw (B)--(D)--(C)--(A)--(B)--(C)--(D)--(A);
	\fill[red, opacity=.7] (C) circle (3pt);
	\fill[red, opacity=.7] (B) circle (3pt);
	\fill[red, opacity=.7] (E) circle (3pt);		
	\draw[line width=2pt, red, opacity=.7] (C)--(B);
\end{tikzpicture}$$

Even though, unlike the case in higher dimensions, there are only finitely many combinatorial types of Alexander moves, the number of edges involved in the state following subdivision of an edge makes using Alexander moves in state-sum constructions unwieldy.  We therefore wish to reduce the problem of PL homeomorphism of curve-surface pairs to Pachner moves, or rather the extended Pachner (bistellar) moves
of \cite{CY}.

It is again trivial by exhaustive consideration of cases to identify the flag-like extended Pachner moves.  There are seven distinct inverse pairs, six of flag-like Pachner moves, and one flag-like extended Pachner move.  We depict them below:

$$
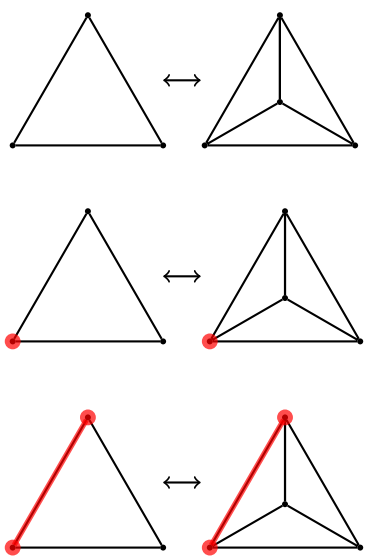
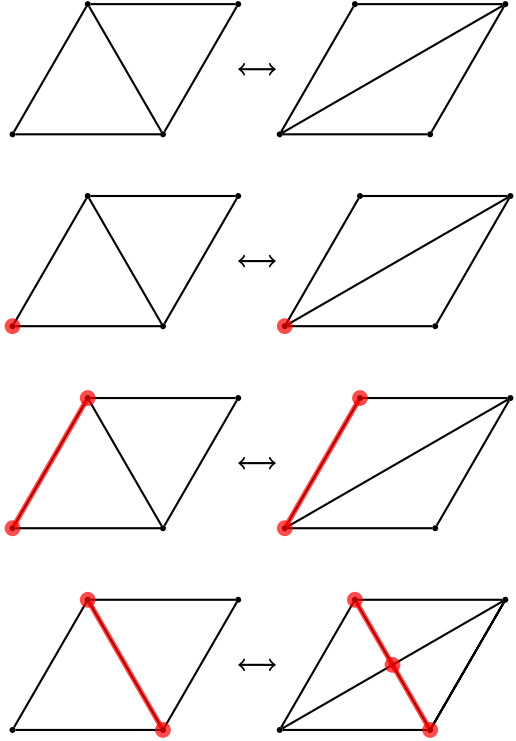
\begin{tikzpicture}[thick] 
	\node[circle, fill, inner sep=.8pt, outer sep=0pt] (A) at (0,0){};
	\node[circle, fill, inner sep=.8pt, outer sep=0pt] (B) at (1,1.73205){};
	\node[circle, fill, inner sep=.8pt, outer sep=0pt] (C) at (2,0){};
	\draw (A)--(B)--(C)--(A);
	\draw [<->] (2,.866)--(2.5,.866);
\end{tikzpicture}
\begin{tikzpicture}[thick] 
	\node[circle, fill, inner sep=.8pt, outer sep=0pt] (A) at (0,0){};
	\node[circle, fill, inner sep=.8pt, outer sep=0pt] (B) at (1,1.73205){};
	\node[circle, fill, inner sep=.8pt, outer sep=0pt] (C) at (2,0){};
	\node[circle, fill, inner sep=.8pt, outer sep=0pt] (D) at (1,.57735){};
	\draw (A)--(B)--(D)--(A)--(C)--(D)--(B)--(C);
\end{tikzpicture}$$

$$\begin{tikzpicture}[thick] 
	\node[circle, fill, inner sep=.8pt, outer sep=0pt] (A) at (0,0){};
	\node[circle, fill, inner sep=.8pt, outer sep=0pt] (B) at (1,1.73205){};
	\node[circle, fill, inner sep=.8pt, outer sep=0pt] (C) at (2,0){};
	\draw (A)--(B)--(C)--(A);
	\fill[red, opacity=.7] (A) circle (3pt);
	\draw [<->] (2,.866)--(2.5,.866);
\end{tikzpicture}
\begin{tikzpicture}[thick] 
	\node[circle, fill, inner sep=.8pt, outer sep=0pt] (A) at (0,0){};
	\node[circle, fill, inner sep=.8pt, outer sep=0pt] (B) at (1,1.73205){};
	\node[circle, fill, inner sep=.8pt, outer sep=0pt] (C) at (2,0){};
	\node[circle, fill, inner sep=.8pt, outer sep=0pt] (D) at (1,.57735){};
	\draw (A)--(B)--(D)--(A)--(C)--(D)--(B)--(C);
	\fill[red, opacity=.7] (A) circle (3pt);
\end{tikzpicture}$$

$$\begin{tikzpicture}[thick] 
	\node[circle, fill, inner sep=.8pt, outer sep=0pt] (A) at (0,0){};
	\node[circle, fill, inner sep=.8pt, outer sep=0pt] (B) at (1,1.73205){};
	\node[circle, fill, inner sep=.8pt, outer sep=0pt] (C) at (2,0){};
	\draw (A)--(B)--(C)--(A);
	\draw[line width=2pt, red, opacity=.7] (A)--(B);
	\fill[red, opacity=.7] (A) circle (3pt);
	\fill[red, opacity=.7] (B) circle (3pt);
	\draw [<->] (2,.866)--(2.5,.866);
\end{tikzpicture}
\begin{tikzpicture}[thick] 
	\node[circle, fill, inner sep=.8pt, outer sep=0pt] (A) at (0,0){};
	\node[circle, fill, inner sep=.8pt, outer sep=0pt] (B) at (1,1.73205){};
	\node[circle, fill, inner sep=.8pt, outer sep=0pt] (C) at (2,0){};
	\node[circle, fill, inner sep=.8pt, outer sep=0pt] (D) at (1,.57735){};
	\draw (A)--(B)--(D)--(A)--(C)--(D)--(B)--(C);
	\draw[line width=2pt, red, opacity=.7] (A)--(B);
	\fill[red, opacity=.7] (A) circle (3pt);
	\fill[red, opacity=.7] (B) circle (3pt);
\end{tikzpicture}$$

$$\begin{tikzpicture}[thick] 
	\node[circle, fill, inner sep=.8pt, outer sep=0pt] (A) at (0,0){};
	\node[circle, fill, inner sep=.8pt, outer sep=0pt] (B) at (1,1.73205){};
	\node[circle, fill, inner sep=.8pt, outer sep=0pt] (C) at (2,0){};
	\node[circle, fill, inner sep=.8pt, outer sep=0pt] (D) at (3,1.73205){};
	\draw (B)--(D)--(C)--(A)--(B)--(C);
	\draw [<->] (3,.866)--(3.5,.866);
\end{tikzpicture}
\begin{tikzpicture}[thick] 
	\node[circle, fill, inner sep=.8pt, outer sep=0pt] (A) at (0,0){};
	\node[circle, fill, inner sep=.8pt, outer sep=0pt] (B) at (1,1.73205){};
	\node[circle, fill, inner sep=.8pt, outer sep=0pt] (C) at (2,0){};
	\node[circle, fill, inner sep=.8pt, outer sep=0pt] (D) at (3,1.73205){};
	\draw (A)--(B)--(D)--(C)--(A)--(D);
\end{tikzpicture}$$

$$\begin{tikzpicture}[thick] 
	\node[circle, fill, inner sep=.8pt, outer sep=0pt] (A) at (0,0){};
	\node[circle, fill, inner sep=.8pt, outer sep=0pt] (B) at (1,1.73205){};
	\node[circle, fill, inner sep=.8pt, outer sep=0pt] (C) at (2,0){};
	\node[circle, fill, inner sep=.8pt, outer sep=0pt] (D) at (3,1.73205){};
	\draw (B)--(D)--(C)--(A)--(B)--(C);
	\fill[red, opacity=.7] (A) circle (3pt);
		\draw [<->] (3,.866)--(3.5,.866);
\end{tikzpicture}
\begin{tikzpicture}[thick] 
	\node[circle, fill, inner sep=.8pt, outer sep=0pt] (A) at (0,0){};
	\node[circle, fill, inner sep=.8pt, outer sep=0pt] (B) at (1,1.73205){};
	\node[circle, fill, inner sep=.8pt, outer sep=0pt] (C) at (2,0){};
	\node[circle, fill, inner sep=.8pt, outer sep=0pt] (D) at (3,1.73205){};
	\draw (A)--(B)--(D)--(C)--(A)--(D);
	\fill[red, opacity=.7] (A) circle (3pt);
\end{tikzpicture}$$

$$\begin{tikzpicture}[thick] 
	\node[circle, fill, inner sep=.8pt, outer sep=0pt] (A) at (0,0){};
	\node[circle, fill, inner sep=.8pt, outer sep=0pt] (B) at (1,1.73205){};
	\node[circle, fill, inner sep=.8pt, outer sep=0pt] (C) at (2,0){};
	\node[circle, fill, inner sep=.8pt, outer sep=0pt] (D) at (3,1.73205){};
	\draw (B)--(D)--(C)--(A)--(B)--(C);
	\draw [<->] (3,.866)--(3.5,.866);
	\fill[red, opacity=.7] (A) circle (3pt);
	\fill[red, opacity=.7] (B) circle (3pt);
	\draw[line width=2pt, red, opacity=.7] (A)--(B);
\end{tikzpicture}
\begin{tikzpicture}[thick] 
	\node[circle, fill, inner sep=.8pt, outer sep=0pt] (A) at (0,0){};
	\node[circle, fill, inner sep=.8pt, outer sep=0pt] (B) at (1,1.73205){};
	\node[circle, fill, inner sep=.8pt, outer sep=0pt] (C) at (2,0){};
	\node[circle, fill, inner sep=.8pt, outer sep=0pt] (D) at (3,1.73205){};
	\draw (A)--(B)--(D)--(C)--(A)--(D);
	\fill[red, opacity=.7] (A) circle (3pt);
	\fill[red, opacity=.7] (B) circle (3pt);
	\draw[line width=2pt, red, opacity=.7] (A)--(B);
\end{tikzpicture}$$

$$\begin{tikzpicture}[thick] 
	\node[circle, fill, inner sep=.8pt, outer sep=0pt] (A) at (0,0){};
	\node[circle, fill, inner sep=.8pt, outer sep=0pt] (B) at (1,1.73205){};
	\node[circle, fill, inner sep=.8pt, outer sep=0pt] (C) at (2,0){};
	\node[circle, fill, inner sep=.8pt, outer sep=0pt] (D) at (3,1.73205){};
	\draw (B)--(D)--(C)--(A)--(B)--(C);
	\draw [<->] (3,.866)--(3.5,.866);
	\fill[red, opacity=.7] (C) circle (3pt);
	\fill[red, opacity=.7] (B) circle (3pt);
	\draw[line width=2pt, red, opacity=.7] (C)--(B);
\end{tikzpicture}
\begin{tikzpicture}[thick] 
	\node[circle, fill, inner sep=.8pt, outer sep=0pt] (A) at (0,0){};
	\node[circle, fill, inner sep=.8pt, outer sep=0pt] (B) at (1,1.73205){};
	\node[circle, fill, inner sep=.8pt, outer sep=0pt] (C) at (2,0){};
	\node[circle, fill, inner sep=.8pt, outer sep=0pt] (D) at (3,1.73205){};
	\node[circle, fill, inner sep=.8pt, outer sep=0pt] (E) at (1.5,.866){};
	\draw (B)--(D)--(C)--(A)--(B)--(C)--(D)--(A);
	\fill[red, opacity=.7] (C) circle (3pt);
	\fill[red, opacity=.7] (B) circle (3pt);
	\fill[red, opacity=.7] (E) circle (3pt);		
	\draw[line width=2pt, red, opacity=.7] (C)--(B);
\end{tikzpicture}$$

The reduction of PL homeomorphism of curve-surface pairs to flag-like extended Pachner moves then follows from Theorem \ref{Alexander_moves_suffice}  and the following:

\begin{proposition} 
\label{extended_Pachner_moves_suffice} Every flag-like Alexander subdivision of a triangulation of a curve-surface pair can be accomplished by a sequence of flag-like extended Pachner moves.
\end{proposition} 

\begin{proof} The first three flag-like Alexander moves are flag-like Pachner moves, so for these the result is trivial.
The next four pairs of flag-like Alexander moves change the triangulation of a cell which intersects the curve in a single simplex, lying in the boundary of the cell.  It is clear in this circumstance that any Pachner move taking place inside the cell must be flag-like, so for these cases the result follows from Pachner's proof \cite{P}.
The tenth flag-like Alexander move is an extended Pachner move, so for this the result is again trivial.

Sequences of flag-like Pachner moves implementing the eighth and ninth flag-like Alexander moves are depicted below:

$$
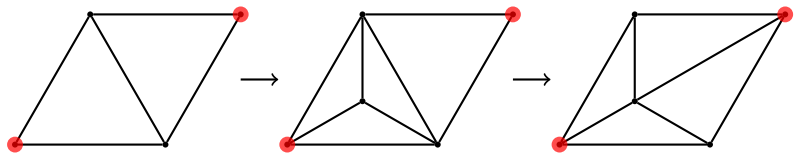
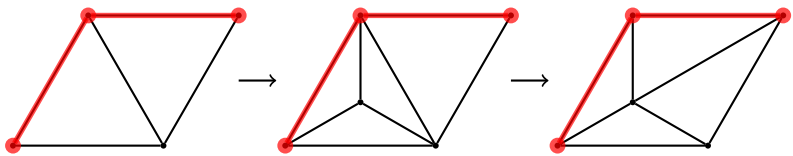
\begin{tikzpicture}[thick] 
	\node[circle, fill, inner sep=.8pt, outer sep=0pt] (A) at (0,0){};
	\node[circle, fill, inner sep=.8pt, outer sep=0pt] (B) at (1,1.73205){};
	\node[circle, fill, inner sep=.8pt, outer sep=0pt] (C) at (2,0){};
	\node[circle, fill, inner sep=.8pt, outer sep=0pt] (D) at (3,1.73205){};
	\draw (B)--(D)--(C)--(A)--(B)--(C);
	\draw [->] (3,.866)--(3.5,.866);
	\fill[red, opacity=.7] (A) circle (3pt);
	\fill[red, opacity=.7] (D) circle (3pt);
\end{tikzpicture}
\begin{tikzpicture}[thick] 
	\node[circle, fill, inner sep=.8pt, outer sep=0pt] (A) at (0,0){};
	\node[circle, fill, inner sep=.8pt, outer sep=0pt] (B) at (1,1.73205){};
	\node[circle, fill, inner sep=.8pt, outer sep=0pt] (C) at (2,0){};
	\node[circle, fill, inner sep=.8pt, outer sep=0pt] (D) at (3,1.73205){};
	\node[circle, fill, inner sep=.8pt, outer sep=0pt] (Z) at (1,.57735){};
	\draw (B)--(D)--(C)--(A)--(B)--(Z)--(A);
	\draw (B)--(C)--(Z);
	\draw [->] (3,.866)--(3.5,.866);
	\fill[red, opacity=.7] (A) circle (3pt);
	\fill[red, opacity=.7] (D) circle (3pt);
\end{tikzpicture}
\begin{tikzpicture}[thick] 
	\node[circle, fill, inner sep=.8pt, outer sep=0pt] (A) at (0,0){};
	\node[circle, fill, inner sep=.8pt, outer sep=0pt] (B) at (1,1.73205){};
	\node[circle, fill, inner sep=.8pt, outer sep=0pt] (C) at (2,0){};
	\node[circle, fill, inner sep=.8pt, outer sep=0pt] (D) at (3,1.73205){};
	\node[circle, fill, inner sep=.8pt, outer sep=0pt] (Z) at (1,.57735){};
	\draw (B)--(D)--(C)--(A)--(B)--(Z)--(A);
	\draw (C)--(Z)--(D);
	\fill[red, opacity=.7] (A) circle (3pt);
	\fill[red, opacity=.7] (D) circle (3pt);
\end{tikzpicture}$$

$$\begin{tikzpicture}[thick] 
	\node[circle, fill, inner sep=.8pt, outer sep=0pt] (A) at (0,0){};
	\node[circle, fill, inner sep=.8pt, outer sep=0pt] (B) at (1,1.73205){};
	\node[circle, fill, inner sep=.8pt, outer sep=0pt] (C) at (2,0){};
	\node[circle, fill, inner sep=.8pt, outer sep=0pt] (D) at (3,1.73205){};
	\draw (B)--(D)--(C)--(A)--(B)--(C);
	\draw [->] (3,.866)--(3.5,.866);
	\fill[red, opacity=.7] (A) circle (3pt);
	\fill[red, opacity=.7] (B) circle (3pt);
	\fill[red, opacity=.7] (D) circle (3pt);
	\draw[line width=2pt, red, opacity=.7] (A)--(B)--(D);
\end{tikzpicture}
\begin{tikzpicture}[thick] 
	\node[circle, fill, inner sep=.8pt, outer sep=0pt] (A) at (0,0){};
	\node[circle, fill, inner sep=.8pt, outer sep=0pt] (B) at (1,1.73205){};
	\node[circle, fill, inner sep=.8pt, outer sep=0pt] (C) at (2,0){};
	\node[circle, fill, inner sep=.8pt, outer sep=0pt] (D) at (3,1.73205){};
	\node[circle, fill, inner sep=.8pt, outer sep=0pt] (Z) at (1,.57735){};
	\draw (B)--(D)--(C)--(A)--(B)--(Z)--(A);
	\draw (B)--(C)--(Z);
	\draw [->] (3,.866)--(3.5,.866);
	\fill[red, opacity=.7] (A) circle (3pt);
	\fill[red, opacity=.7] (B) circle (3pt);
	\fill[red, opacity=.7] (D) circle (3pt);
	\draw[line width=2pt, red, opacity=.7] (A)--(B)--(D);
\end{tikzpicture}
\begin{tikzpicture}[thick] 
	\node[circle, fill, inner sep=.8pt, outer sep=0pt] (A) at (0,0){};
	\node[circle, fill, inner sep=.8pt, outer sep=0pt] (B) at (1,1.73205){};
	\node[circle, fill, inner sep=.8pt, outer sep=0pt] (C) at (2,0){};
	\node[circle, fill, inner sep=.8pt, outer sep=0pt] (D) at (3,1.73205){};
	\node[circle, fill, inner sep=.8pt, outer sep=0pt] (Z) at (1,.57735){};
	\draw (B)--(D)--(C)--(A)--(B)--(Z)--(A);
	\draw (C)--(Z)--(D);
	\fill[red, opacity=.7] (A) circle (3pt);
	\fill[red, opacity=.7] (B) circle (3pt);
	\fill[red, opacity=.7] (D) circle (3pt);
	\draw[line width=2pt, red, opacity=.7] (A)--(B)--(D);
\end{tikzpicture}$$

\end{proof}

\section{Colorings of triangulated surface-curve pairs}
\subsection{Edge labels and orientations}
Each edge of a 2-simplex either does not intersect the curve, intersects the curve at only one point, or is totally contained in the curve.  The restriction to flag-like triangulations ensures that the fourth possibility, intersecting the curve in exactly two points (the end points), does not arise.   As in the combinatorial constructions of defect-free finite gauge group Dijkgraaf-Witten theory (cf. \cite{W}),  edges not intersecting the curve are colored with elements of a finite group $G$.  Edges intersecting the curve at one point are colored with elements of a set $X$, and edges lying on the curve are colored with elements of a finite group $H$:
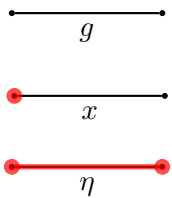
$$\begin{tikzpicture}[thick]
		\node[circle, fill, inner sep=.8pt, outer sep=0pt] (A) at (0,0){};
		\node[circle, fill, inner sep=.8pt, outer sep=0pt] (B) at (2,0){};
		\draw (A)--(B) node[font=\normalsize, midway, below] {$g$};
\end{tikzpicture}$$
$$\begin{tikzpicture}[thick]
		\node[circle, fill, inner sep=.8pt, outer sep=0pt] (A) at (0,0){};
		\node[circle, fill, inner sep=.8pt, outer sep=0pt] (B) at (2,0){};
		\draw (A)--(B) node[font=\normalsize, midway, below] {$x$};
		\fill [red, opacity=.7] (A) circle (3pt);
\end{tikzpicture}$$
$$\begin{tikzpicture}[thick]
		\node[circle, fill, inner sep=.8pt, outer sep=0pt] (A) at (0,0){};
		\node[circle, fill, inner sep=.8pt, outer sep=0pt] (B) at (2,0){};
		\draw (A)--(B) node[font=\normalsize, midway, below] {$\eta$};
		\draw[line width=2pt, red, opacity=.7] (A) -- (B);
		\fill [red, opacity=.7] (A) circle (3pt);
		\fill [red, opacity=.7] (B) circle (3pt);
\end{tikzpicture}$$
where $g\in G$, $x\in X$, and $\eta\in H$.

The set $X$ will, moreover, be equipped with a right action of the group $G$ and a left action of the group $H$.

As in \cite{Y}, edge orientations are determined by enumerating the vertices of the triangulation and orienting the edges from the earlier end point to the later.  The vertices lying on the curve are always enumerated first.  Then, the remaining vertices are enumerated.  

It follows from this ordering that edges with only one vertex on the curve will always be oriented away from the vertex intersecting the curve:
$$\begin{tikzpicture}[thick]
		\node[circle, fill, inner sep=.8pt, outer sep=0pt] (A) at (0,0){};
		\node[circle, fill, inner sep=.8pt, outer sep=0pt] (B) at (2,0){};
		\draw [midarrow={>}] (A)--(B);
		\fill [red, opacity=.7] (A) circle (3pt);
\end{tikzpicture}$$

It also follows that every 2-simplex will have two edges oriented in the same direction, with the third edge oriented oppositely.  Additionally, the orientation of those two edges give each 2-simplex its own orientation, which may either agree with or be opposite the given orientation of the surface.  By convention, we will always depict portions of triangulations as viewed from the side of the surface from which counterclockwise corresponds to the given orientation of the surface.

\subsection{Admissible colorings of  2-simplexes}
Consider the three possible flag-like 2-simplexes.
$$\setlength{\tabcolsep}{1.5em}{
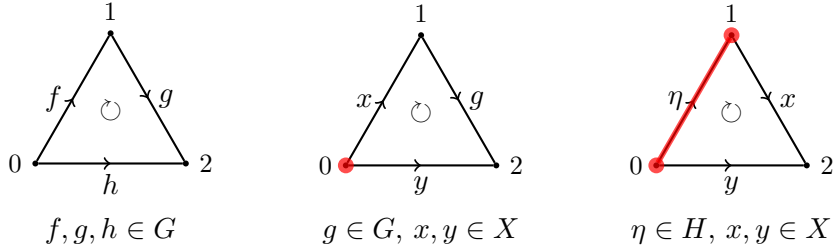
\begin{tabular}{c c c}
	\begin{tikzpicture}[thick] 
		\node[circle, label=left:{\small0}, fill, inner sep=.8pt, outer sep=0pt] (A) at (0,0){};
		\node[circle, label=above:{\small1}, fill, inner sep=.8pt, outer sep=0pt] (B) at (1,1.73205){};
		\node[circle, label=right:{\small2}, fill, inner sep=.8pt, outer sep=0pt] (C) at (2,0){};
		\draw [midarrow={>}] (A)--(B) node[midway, left]{$f$};
		\draw [midarrow={>}] (B)--(C) node[midway, right]{$g$};
		\draw [midarrow={>}] (A)--(C) node[midway, below]{$h$};	
		\node at (1,.7){$\circlearrowright$};
	\end{tikzpicture}
	&
	\begin{tikzpicture}[thick]
		\node[circle, label=left:{\small0}, fill, inner sep=.8pt, outer sep=0pt] (A) at (0,0){};
		\node[circle, label=above:{\small1}, fill, inner sep=.8pt, outer sep=0pt] (B) at (1,1.73205){};
		\node[circle, label=right:{\small2}, fill, inner sep=.8pt, outer sep=0pt] (C) at (2,0){};
		\draw [midarrow={>}] (A)--(B) node[midway, left]{$x$};
		\draw [midarrow={>}] (B)--(C) node[midway, right]{$g$};
		\draw [midarrow={>}] (A)--(C) node[midway, below]{$y$};
		\fill [red, opacity=.7] (A) circle (3pt);
		\node at (1,.7){$\circlearrowright$};
	\end{tikzpicture}
	&
	\begin{tikzpicture}[thick]
		\node[circle, label=left:{\small0}, fill, inner sep=.8pt, outer sep=0pt] (A) at (0,0){};
		\node[circle, label=above:{\small1}, fill, inner sep=.8pt, outer sep=0pt] (B) at (1,1.73205){};
		\node[circle, label=right:{\small2}, fill, inner sep=.8pt, outer sep=0pt] (C) at (2,0){};
		\draw [midarrow={>}] (A)--(B) node[midway, left]{$\eta$};
		\draw [midarrow={>}] (B)--(C) node[midway, right]{$x$};
		\draw [midarrow={>}] (A)--(C) node[midway, below]{$y$};
		\draw[line width=2pt, red, opacity=.7] (A)--(B);
		\fill [red, opacity=.7] (A) circle (3pt);
		\fill [red, opacity=.7] (B) circle (3pt);
		\node at (1,.7){$\circlearrowright$};
	\end{tikzpicture}\\
	$f, g, h\in G$ & $g\in G$, $x, y\in X$ & $\eta\in H$, $x, y\in X$
\end{tabular}}$$

In these figures and those following, the numbering of the vertices is not intended to indicate their absolute numbering in the enumeration of the vertices, but their relative positions in that enumeration.  In each figure vertices will be numbered with an initial segment of the natural numbers beginning with zero.

The notion of admissiblity will be the most natural generalization of that used in the combinatorial construction of finite-gauge group Dijkgraaf-Witten theory:  for simplexes not intersecting the curve, as expected, the 2-simplex is admissibly labeled if the product of the group elements labeling the edges oriented in the same direction is the group element labeling the third edge, i.e. in the notation of the figures above, if $fg=h$.
A 2-simplex with only one vertex intersecting the curve is admissibly labeled if $x\cdot g=y$, where $\cdot: X\times G\to X$ denotes the right action of $G$ on $X$.  Similarly, a 2-simplex with an edge entirely on the curve is admissibly labeled if $\eta\cdot x=y$ for the left group action $\cdot: H\times X\to X$.

\subsection{Admissible colorings of triangulations and extended Pachner moves}

\begin{definition} Let $G$ and $H$ be two finite groups, and $X$ as set with a right $G$-action and a left $H$-action which commute (in the natural sense which looks like a restricted associativity condtion: $h\cdot(x\cdot g) = (h\cdot x)\cdot g$). An {\em$(H,X,G)$-coloring} of a triangulation $\mathcal T$ of a surface-curve pair $(\Sigma \supset C)$ is a function $\lambda:{\mathcal T}_1\rightarrow G \coprod X \coprod H$ such that the value of every edge not intersecting $C$ lies in $G$, the value of every edge intersecting $C$ in a single point lies in $X$, and the value of every edge of $C$ lies in $H$.  Here ${\mathcal T}_k$ denotes the set of $k$-simplexes in the triangulation $\mathcal T$.  

An $(H,X,G)$-coloring is {\em admissible} if every 2-simplex is admissible as described above.
\end{definition}  

To determine both what is the appropriate generalization of a 2-cocycle on the gauge group and what sort of normalization procedure will give a topological invariant of surface-curve pairs, it is necessary to understand the behavior of admissible colorings under flag-like extended Pachner moves.

For each of the three types of triangles, a 1-3 subdivision move will introduce a new vertex and three new edges, with the edge between the new vertex and the last numbered vertex of the unsubdivided triangle disjoint from the curve.  An admissible coloring of the original triangle, together with the choice of an element of $G$ with which to color this edge, then uniquely determines admissible colorings of the three smaller triangles of the subdivision

$$
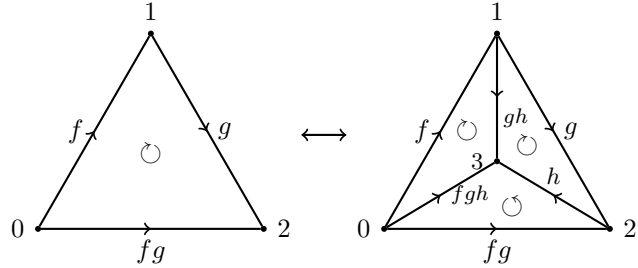
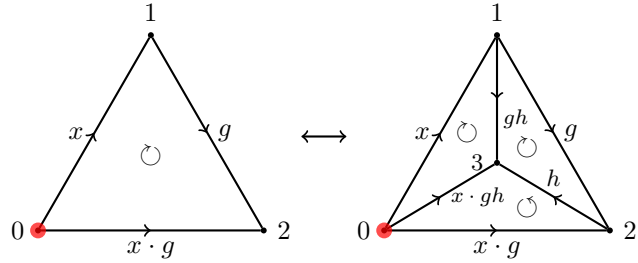
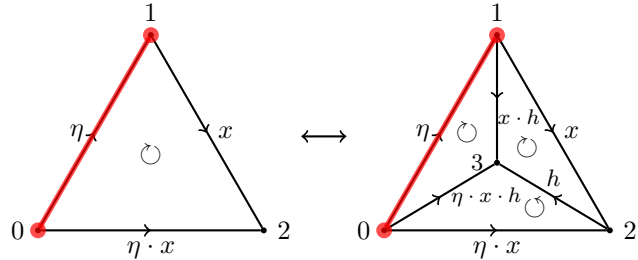
\begin{tikzpicture}[thick]
	\node[circle, label=left:{\small0}, fill, inner sep=.8pt, outer sep=0pt] (A) at (0,0){};
	\node[circle, label=above:{\small1}, fill, inner sep=.8pt, outer sep=0pt] (B) at (1.5,2.598075){};
	\node[circle, label=right:{\small2}, fill, inner sep=.8pt, outer sep=0pt] (C) at (3,0){};
	\begin{scope}
		\draw [midarrow={>}] (A)--(B) node[font=\small, midway,left]{$f$};
		\draw [midarrow={>}] (B)--(C) node[font=\small, midway,right]{$g$};
		\draw [midarrow={>}] (A)--(C) node[font=\small, midway,below]{$fg$};
		\node at (1.5,1){$\circlearrowright$};		
		\draw [<->] (3.5,1.25)--(4.1,1.25);
	\end{scope}
\end{tikzpicture}
\begin{tikzpicture}[thick]
	\node[circle, label=left:{\small0}, fill, inner sep=.8pt, outer sep=0pt] (A) at (0,0){};
	\node[circle, label=above:{\small1}, fill, inner sep=.8pt, outer sep=0pt] (B) at (1.5,2.598075){};
	\node[circle, label=right:{\small2}, fill, inner sep=.8pt, outer sep=0pt] (C) at (3,0){};
	\node[circle, label=left: {\footnotesize3}, fill, inner sep=.8pt, outer sep=0pt] (D) at (1.5,0.9){};
	\begin{scope}
		\draw [midarrow={>}] (A)--(B) node[font=\small, midway,left]{$f$};
		\draw [midarrow={>}] (B)--(C) node[font=\small, midway,right]{$g$};
		\draw [midarrow={>}] (A)--(C) node[font=\small, midway,below]{$fg$};
		\draw [midarrow={>}] (B)--(D) node[font=\scriptsize, midway, below]{$\;\;\;\;\;\;gh$};
		\draw [midarrow={>}] (C)--(D) node[font=\footnotesize, midway, above]{$h$};
		\draw [midarrow={>}] (A)--(D) node[font=\scriptsize, midway, right]{$fgh$};
		\node at (1.1,1.3){$\circlearrowright$};
		\node at (1.9,1.1){$\circlearrowright$};
		\node at (1.7,0.3){$\circlearrowleft$};	
	\end{scope}
\end{tikzpicture}$$

where $f, g, h\in G$;

$$\begin{tikzpicture}[thick]
	\node[circle, label=left:{\small0}, fill, inner sep=.8pt, outer sep=0pt] (A) at (0,0){};
	\node[circle, label=above:{\small1}, fill, inner sep=.8pt, outer sep=0pt] (B) at (1.5,2.598075){};
	\node[circle, label=right:{\small2}, fill, inner sep=.8pt, outer sep=0pt] (C) at (3,0){};
	\begin{scope}
		\draw [midarrow={>}] (A)--(B) node[font=\small, midway,left]{$x$};
		\draw [midarrow={>}] (B)--(C) node[font=\small, midway,right]{$g$};
		\draw [midarrow={>}] (A)--(C) node[font=\small, midway,below]{$x\cdot g$};
		\fill [red, opacity=.7] (A) circle (3pt);
		\node at (1.5,1){$\circlearrowright$};				
		\draw [<->] (3.5,1.25)--(4.1,1.25);
	\end{scope}
\end{tikzpicture}
\begin{tikzpicture}[thick]
	\node[circle, label=left:{\small0}, fill, inner sep=.8pt, outer sep=0pt] (A) at (0,0){};
	\node[circle, label=above:{\small1}, fill, inner sep=.8pt, outer sep=0pt] (B) at (1.5,2.598075){};
	\node[circle, label=right:{\small2}, fill, inner sep=.8pt, outer sep=0pt] (C) at (3,0){};
	\node[circle, label=left: {\footnotesize3}, fill, inner sep=.8pt, outer sep=0pt] (D) at (1.5,0.9){};
	\begin{scope}
		\draw [midarrow={>}] (A)--(B) node[font=\small, midway,left]{$x$};
		\draw [midarrow={>}] (B)--(C) node[font=\small, midway,right]{$g$};
		\draw [midarrow={>}] (A)--(C) node[font=\small, midway,below]{$x\cdot g$};
		\draw [midarrow={>}] (B)--(D) node[font=\scriptsize, midway, below]{$\;\;\;\;\;\;gh$};
		\draw [midarrow={>}] (C)--(D) node[font=\footnotesize, midway, above]{$h$};
		\draw [midarrow={>}] (A)--(D) node[font=\scriptsize, midway, right]{$x\cdot gh$};
		\fill [red, opacity=.7] (A) circle (3pt);
		\node at (1.1,1.3){$\circlearrowright$};
		\node at (1.9,1.1){$\circlearrowright$};
		\node at (1.9,0.3){$\circlearrowleft$};					
	\end{scope}
\end{tikzpicture}$$

where $x\in X$ and $g, h\in G$;

$$\begin{tikzpicture}[thick]
	\node[circle, label=left:{\small0}, fill, inner sep=.8pt, outer sep=0pt] (A) at (0,0){};
	\node[circle, label=above:{\small1}, fill, inner sep=.8pt, outer sep=0pt] (B) at (1.5,2.598075){};
	\node[circle, label=right:{\small2}, fill, inner sep=.8pt, outer sep=0pt] (C) at (3,0){};
	\begin{scope}
		\draw [midarrow={>}] (A)--(B) node[font=\small, midway,left]{$\eta$};
		\draw [midarrow={>}] (B)--(C) node[font=\small, midway,right]{$x$};
		\draw [midarrow={>}] (A)--(C) node[font=\small, midway,below]{$\eta\cdot x$};
		\draw[line width=2pt, red, opacity=.7] (A)--(B);		
		\fill [red, opacity=.7] (A) circle (3pt);
		\fill [red, opacity=.7] (B) circle (3pt);
		\node at (1.5,1){$\circlearrowright$};			
		\draw [<->] (3.5,1.25)--(4.1,1.25);
	\end{scope}
\end{tikzpicture}
\begin{tikzpicture}[thick]
	\node[circle, label=left:{\small0}, fill, inner sep=.8pt, outer sep=0pt] (A) at (0,0){};
	\node[circle, label=above:{\small1}, fill, inner sep=.8pt, outer sep=0pt] (B) at (1.5,2.598075){};
	\node[circle, label=right:{\small2}, fill, inner sep=.8pt, outer sep=0pt] (C) at (3,0){};
	\node[circle, label=left: {\footnotesize3}, fill, inner sep=.8pt, outer sep=0pt] (D) at (1.5,0.9){};
	\begin{scope}
		\draw [midarrow={>}] (A)--(B) node[font=\small, midway,left]{$\eta$};
		\draw [midarrow={>}] (B)--(C) node[font=\small, midway,right]{$x$};
		\draw [midarrow={>}] (A)--(C) node[font=\small, midway,below]{$\eta\cdot x$};
		\draw [midarrow={>}] (B)--(D) node[font=\scriptsize, midway, below]{$\;\;\;\;\;\;\;x\cdot h$};
		\draw [midarrow={>}] (C)--(D) node[font=\footnotesize, midway, above]{$h$};
		\draw [midarrow={>}] (A)--(D) node[font=\scriptsize, midway, right]{$\eta\cdot x\cdot h$};
		\draw[line width=2pt, red, opacity=.7] (A)--(B);
		\fill [red, opacity=.7] (A) circle (3pt);
		\fill [red, opacity=.7] (B) circle (3pt);	
		\node at (1.1,1.3){$\circlearrowright$};
		\node at (1.9,1.1){$\circlearrowright$};
		\node at (2,0.3){$\circlearrowleft$};							
	\end{scope}
\end{tikzpicture}$$

where $\eta\in H$, $g, h\in G$, and $x\in X$.

Observe in each case that the well-defininedness of one of the labels depends on a generalization of associativity -- for the triangle not incident with the curve associativity within $G$, for the triangle intersecting the curve in a single vertex the associativity condition in the definition of a right group action, and for the triangle with an edge lying on the curve, the commutativity of the two group actions.  In each case we have inserted the new vertex after the vertices of the original triangle in the new ordering.  It is an easy, though tedious exercise, left to the reader, to verify that the same conclusion holds if the new vertex is placed anywhere else in the ordering of the vertices not on the curve.

Another easy exercise, also left to the reader, is to verify that if the edges of the initial state of a 2-2 move are admissibly colored, there is a unique admissible coloring of the final state agreeing with the given coloring on the edges which are not changed by the 2-2 move.  Again, one case uses each analog of associativity.

Finally, we consider the behavior of admissible colorings under the extended Pachner move subdividing an edge of the curve:

$$
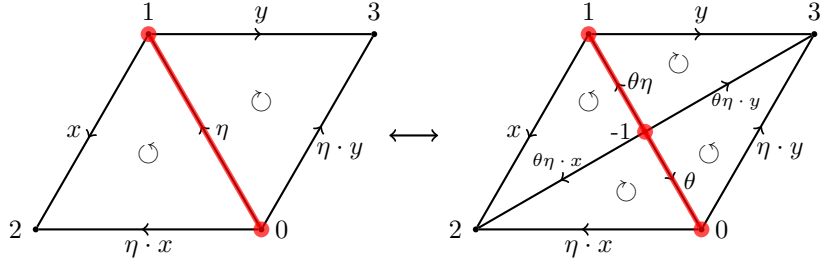
\begin{tikzpicture}[thick] 
	\node[circle, label=left:{\small2}, fill, inner sep=.8pt, outer sep=0pt] (A) at (0,0){};
	\node[circle, label=above:{\small1}, fill, inner sep=.8pt, outer sep=0pt] (B) at (1.5,2.598075){};
	\node[circle, label=right:{\small0}, fill, inner sep=.8pt, outer sep=0pt] (C) at (3,0){};
	\node[circle, label=above:{\small3}, fill, inner sep=.8pt, outer sep=0pt] (D) at (4.5,2.598075){};
	\begin{scope}
		\draw [midarrow={<}] (A)--(B) node[font=\small, midway, left]{$x$};
		\draw [midarrow={>}] (B)--(D) node[font=\small, midway, above]{$y$};
		\draw [midarrow={<}] (D)--(C) node[font=\small, midway, below]{$\;\;\;\;\;\;\eta\cdot y$};
		\draw [midarrow={<}] (A)--(C) node[font=\small, midway, below]{$\eta\cdot x$};
		\draw [midarrow={<}] (B)--(C) node[font=\footnotesize, midway, right]{$\eta$};
		\draw [<->] (4.7,1.25)--(5.35,1.25);
		\fill[red, opacity=.7] (B) circle (3pt);
		\fill[red, opacity=.7] (C) circle (3pt);
		\draw[line width=2pt, red, opacity=.7] (B)--(C);
	    	\node at (1.5,1){$\circlearrowleft$};
		\node at (3,1.7){$\circlearrowright$};		
	\end{scope}
\end{tikzpicture}
\begin{tikzpicture}[thick] 
	\node[circle, label=left:{\small2}, fill, inner sep=.8pt, outer sep=0pt] (A) at (0,0){};
	\node[circle, label=above:{\small1}, fill, inner sep=.8pt, outer sep=0pt] (B) at (1.5,2.598075){};
	\node[circle, label=right:{\small0}, fill, inner sep=.8pt, outer sep=0pt] (C) at (3,0){};
	\node[circle, label=above:{\small3}, fill, inner sep=.8pt, outer sep=0pt] (D) at (4.5,2.598075){};
	\node[circle, label=left:{\footnotesize-1}, fill, inner sep=.8pt, outer sep=0pt] (E) at (2.25,1.299){};
	\begin{scope}
		\draw [midarrow={<}] (A)--(B) node[font=\small, midway, left]{$x$};
		\draw [midarrow={>}] (B)--(D) node[font=\small, midway, above]{$y$};
		\draw [midarrow={<}] (D)--(C) node[font=\small, midway, below]{$\;\;\;\;\;\;\eta\cdot y$};
		\draw [midarrow={<}] (A)--(C) node[font=\small, midway, below]{$\eta\cdot x$};
		\draw [midarrow={>}] (E)--(C) node[font=\footnotesize, midway, right]{$\theta$};
		\draw [midarrow={>}] (E)--(B) node[font=\footnotesize, midway, right]{$\theta\eta$};	
		\draw [midarrow={>}] (E)--(D) node[font=\tiny, midway, below]{$\;\;\theta\eta\cdot y$};
		\draw [midarrow={>}] (E)--(A) node[font=\tiny, midway, above]{$\theta\eta\cdot x\;$};
		\fill[red, opacity=.7] (B) circle (3pt);
		\fill[red, opacity=.7] (C) circle (3pt);
		\fill[red, opacity=.7] (E) circle (3pt);		
		\draw[line width=2pt, red, opacity=.7] (B)--(C);
	     	\node at (1.5,1.7){$\circlearrowleft$};
		\node at (2,0.5){$\circlearrowright$};
		\node at (2.7,2.2){$\circlearrowright$};
		\node at (3.1,1){$\circlearrowleft$};		
	\end{scope}
\end{tikzpicture}$$

Again the admissible coloring of the welded state, together with a choice, this time of an element $\theta \in H$, determines a unique admissible coloring of the subdivided state.  Note, here we chose to introduce the new vertex into the ordering of vertices earlier than all the vertices in the welded state.  Again we leave it to the reader to verify that the same conclusion holds if the new vertex is inserted anywhere else in the ordering of the vertices on the curve.

From the foregoing discussion it is essentially immediate that we have a generalization of the invariants of surfaces arising from ``untwisted'' 2-dimensional Dijkgraaf-Witten theory to surface-curve pairs:

\begin{theorem} \label{untwisted}
Let $G$ and $H$ be finite groups, and $X$ a set equipped commuting group actions of $G$ on the right and $H$ on the left.  For any flag-like triangulation of a surface-curve pair $\Sigma \supset C$, $\mathcal T$, with vertices ordered as described above, let $\kappa_{(H,X,G)}({\mathcal T})$ denote the number of admissible $(H,X,G)$-colorings of $\mathcal T$, and ${\mathcal T}_n^k$ denote the set of $n$-simplexes of $\mathcal T$ with $k$ vertices lying on the curve, then the quantity

\[ Z_{(H,X,G)}(\Sigma \supset C) := |G|^{-|{\mathcal T}_0^0|} |H|^{-|{\mathcal T}_0^1|} \kappa_{(H,X,G)}({\mathcal T}) \]

\noindent is independent of the triangulation and the ordering on the vertices, and thus a topological invariant of the surface-curve pair $\Sigma \supset C$.

\end{theorem}

\begin{proof}
Having observed above that extended Pachner moves which introduce a vertex on (resp. off) the curve replace an admissible coloring with $|H|$ (resp. $|G|$) admissible colorings, this follows immediately from Proposition \ref{extended_Pachner_moves_suffice}.  The one thing to note is that the independence from the ordering of the vertices follows from the same proposition by the simple trick of performing (flag-like) extended Pachner moves to obtain a triangulation with no vertices in common with the original, then performing the inverses of the same moves in reverse sequence, but inserting the original vertices into the ordering to agree with some other desired ordering, rather than the given one.
\end{proof}

Of course it is easy to express the invariant of Theorem \ref{untwisted} as a state-sum, since $\kappa_{(H,X,G)}({\mathcal T}) = \sum_{\lambda} \prod_{\sigma \in {\mathcal T}_2} \chi(\lambda, \sigma))$,
where $\lambda$ ranges over all $(H,X,G)$-colorings and 

\[ \chi(\lambda, \sigma) = \left\{\begin{array}{ll} 1 & \mbox{\rm if $\lambda$ colors $\sigma$ admissibly} \\ 0 & \mbox{\rm otherwise} \end{array} \right.\]

\section{Generalized 2-cocycles}

The invariant just described is the natural generalization of ``untwisted'' 2-dimensional Dijkgraaf-Witten theory from surfaces to surface-curve pairs, but in general $n$-dimensional Dijkgraaf-Witten theory is described by the (finite) gauge group, together with an $n$-cocycle valued in the multiplicative group of a chosen ground field, $K$, which provides local state values for $n$-simplexes, explicitly in the case of $n=2$ a function $\alpha: G \times G \rightarrow K^\times$ satisfying

\begin{equation} \label{cond1}
\forall f,g,h \in G\;\; \alpha(f,g)\alpha^{-1}(f, gh)\alpha(fg, h)\alpha^{-1}(g,h)=1
\end{equation}

The corresponding invariant of oriented surfaces is then given in terms of a triangulation with ordered vertices $\mathcal T$ by

\[ Z_{(G,\alpha)}(\Sigma) = |G|^{-|{\mathcal T}_0|} \sum_\lambda \prod_{\sigma \in {\mathcal T}_2} \alpha^{\epsilon(\sigma)}(\tilde{\lambda}(\sigma)) \]

\noindent where $\lambda$ ranges over all admissible colorings, $\tilde{\lambda}(\sigma)$ is the 
pair of elements of $G$ coloring the edges 01 and 12 of $\sigma$, and 

\[ \epsilon(\sigma) = \left\{\begin{array}{ll} 1 & \mbox{\rm if the orientation of $\sigma$ induced by the ordering}\\ & \mbox{\rm \hspace*{5mm} agrees with the orientation of the surface} \\ -1 & \mbox{\rm if the orientation of $\sigma$ induced by the ordering}\\ & \mbox{\rm \hspace*{5mm} is opposite the orientation of the surface} \end{array} \right.\]

It is an easy exercise, similar to calculations we are about to give for instances of our local state values on triangles intersecting the curve, to verify first that the 2-cocycle condition is precisely the condition needed to ensure that when the initial state of a 2-2 Pachner move is admissibly colored, and the final state is given the unique admissible coloring agreeing with the given one on the unmodified edges, the products of the quantities $\alpha^{\epsilon(\sigma)}(\tilde{\lambda}(\sigma))$ over the initial and final pairs of triangles are equal, and second that the 2-cocycle condition also implies that the product of $\alpha^{\epsilon(\sigma)}(\tilde{\lambda}(\sigma))$ over the three triangles of the subdivided state of a 1-3 Pachner move is equal to to the value of $\alpha^{\epsilon(\sigma)}(\tilde{\lambda}(\sigma))$ on the single unsubdivided triangle, whenever $\lambda$ is taken to be an admissible coloring of the triangle in the latter case and any of the admissible colorings induced by $\lambda$ and a choice of label for a new edge in the former.

We, of course, require local state values for triangles with a vertex or an edge on the curve, and thus will consider a triple of functions, $(\alpha, \beta, \gamma)$

$$\begin{tabular}{c}
$\alpha: G\times G \to K^\times$\\
$\beta: X\times G\to K^\times$\\
$\gamma: H\times X\to K^\times$
\end{tabular}$$

Plainly the conditions required for invariance under Pachner moves on cells not intersecting the curve are the same as for ordinary Dijkgraaf-Witten theory and require that $\alpha$ be a 2-cocycle on the group $G$.

Now consider a 2-2 Pachner move with a single vertex on the curve:

$$
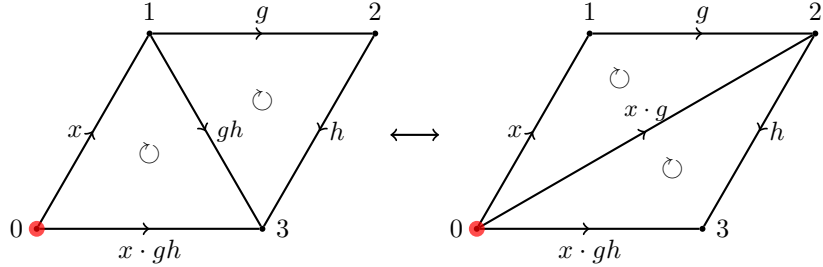
\begin{tikzpicture}[thick] 
	\node[circle, label=left:{\small0}, fill, inner sep=.8pt, outer sep=0pt] (A) at (0,0){};
	\node[circle, label=above:{\small1}, fill, inner sep=.8pt, outer sep=0pt] (B) at (1.5,2.598075){};
	\node[circle, label=right:{\small3}, fill, inner sep=.8pt, outer sep=0pt] (C) at (3,0){};
	\node[circle, label=above:{\small2}, fill, inner sep=.8pt, outer sep=0pt] (D) at (4.5,2.598075){};
	\begin{scope}
		\draw [midarrow={>}] (A)--(B) node[font=\small, midway, left]{$x$};
		\draw [midarrow={>}] (B)--(D) node[font=\small, midway, above]{$g$};
		\draw [midarrow={>}] (D)--(C) node[font=\small, midway, right]{$h$};
		\draw [midarrow={>}] (A)--(C) node[font=\small, midway, below]{$x\cdot gh$};	
		\draw [midarrow={>}] (B)--(C) node[font=\footnotesize, midway, right]{$gh$};
		\node at (1.5,1){$\circlearrowright$};
		\node at (3,1.7){$\circlearrowright$};
		\draw [<->] (4.7,1.25)--(5.35,1.25);
		\fill[red, opacity=.7] (A) circle (3pt);
	\end{scope}
\end{tikzpicture}
\begin{tikzpicture}[thick] 
	\node[circle, label=left:{\small0}, fill, inner sep=.8pt, outer sep=0pt] (A) at (0,0){};
	\node[circle, label=above:{\small1}, fill, inner sep=.8pt, outer sep=0pt] (B) at (1.5,2.598075){};
	\node[circle, label=right:{\small3}, fill, inner sep=.8pt, outer sep=0pt] (C) at (3,0){};
	\node[circle, label=above:{\small2}, fill, inner sep=.8pt, outer sep=0pt] (D) at (4.5,2.598075){};
	\begin{scope}
		\draw [midarrow={>}] (A)--(B) node[font=\small, midway, left]{$x$};
		\draw [midarrow={>}] (B)--(D) node[font=\small, midway, above]{$g$};
		\draw [midarrow={>}] (D)--(C) node[font=\small, midway, right]{$h$};
		\draw [midarrow={>}] (A)--(C) node[font=\small, midway, below]{$x\cdot gh$};	
		\draw [midarrow={>}] (A)--(D) node[font=\footnotesize, midway, above]{$x \cdot g$};
	     	\node at (1.9,2.0){$\circlearrowright$};
		\node at (2.6,0.8){$\circlearrowright$};
		\fill[red, opacity=.7] (A) circle (3pt);
	\end{scope}
\end{tikzpicture}$$

The choice of $x \in X$ and $g,h \in G$ determine unique admissible colorings of both the initial and final state and require the equation

$$\beta(x, gh)\alpha(g, h)=\beta(x, g)\beta(x\cdot g, h)$$

\noindent for the product of the local state values to remain unchanged by the move.  Solving to put all the factors on one side gives a cocycle-type equation 

\begin{equation} \label{cond2}
 \forall x \in X\; g,h\in G\;\; \beta(x, g)\beta^{-1}(x, gh)\beta(x\cdot g, h)\alpha^{-1}(g, h)=1 
\end{equation}

\noindent which also ensures that the product of local state values remains unchanged when an admissible coloring of a triangle with a single vertex on the curve is replaced with the admissible coloring corresponding to any choice of a label from $G$ for a new edge not incident with the curve after a 1-3 subdivision:

$$
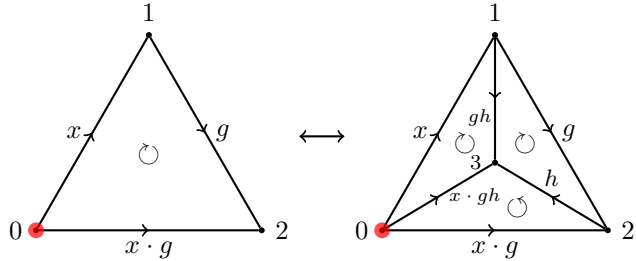
\begin{tikzpicture}[thick] 
	\node[circle, label=left:{\small0}, fill, inner sep=.8pt, outer sep=0pt] (A) at (0,0){};
	\node[circle, label=above:{\small1}, fill, inner sep=.8pt, outer sep=0pt] (B) at (1.5,2.598075){};
	\node[circle, label=right:{\small2}, fill, inner sep=.8pt, outer sep=0pt] (C) at (3,0){};
	\begin{scope}
		\draw [midarrow={>}] (A)--(B) node[font=\small, midway, left]{$x$};
		\draw [midarrow={>}] (B)--(C) node[font=\small, midway, right]{$g$};
		\draw [midarrow={>}] (A)--(C) node[font=\small, midway, below]{$x\cdot g$};
		\draw [<->] (3.5,1.25)--(4.1,1.25);
		\node at (1.5,1){$\circlearrowright$};
		\fill[red, opacity=.7] (A) circle (3pt);
	\end{scope}
\end{tikzpicture}
\begin{tikzpicture}[thick] 
	\node[circle, label=left:{\small0}, fill, inner sep=.8pt, outer sep=0pt] (A) at (0,0){};
	\node[circle, label=above:{\small1}, fill, inner sep=.8pt, outer sep=0pt] (B) at (1.5,2.598075){};
	\node[circle, label=right:{\small2}, fill, inner sep=.8pt, outer sep=0pt] (C) at (3,0){};
	\node[circle, label=left: {\scriptsize 3}, fill, inner sep=.8pt, outer sep=0pt] (D) at (1.5,0.9){};
	\begin{scope}
		\draw [midarrow={>}] (A)--(B) node[font=\small, midway, left]{$x$};
		\draw [midarrow={>}] (B)--(C) node[font=\small, midway, right]{$g$};
		\draw [midarrow={>}] (A)--(C) node[font=\small, midway, below]{$x\cdot g$};
		\draw [midarrow={>}] (C)--(D) node[font=\footnotesize, midway, above]{$h$};
		\draw [midarrow={>}] (B)--(D) node[font=\tiny, midway, below]{$gh\;\;\;\;\;$};
		\draw [midarrow={>}] (A)--(D) node[font=\tiny, midway, right]{$x\cdot gh$};
		\node at (1.1,1.15){$\circlearrowright$};
		\node at (1.8,0.3){$\circlearrowleft$};
		\node at (1.9,1.15){$\circlearrowright$};
		\fill[red, opacity=.7] (A) circle (3pt);
	\end{scope}
\end{tikzpicture}$$

Similarly, considering a 2-2 Pachner move with an edge in the boundary of the cell lying on the curve 

$$
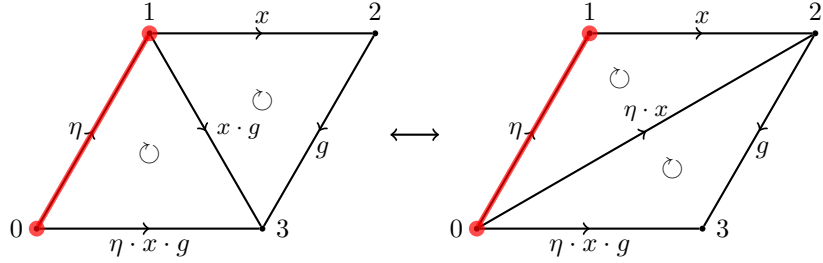
\begin{tikzpicture}[thick] 
	\node[circle, label=left:{\small0}, fill, inner sep=.8pt, outer sep=0pt] (A) at (0,0){};
	\node[circle, label=above:{\small1}, fill, inner sep=.8pt, outer sep=0pt] (B) at (1.5,2.598075){};
	\node[circle, label=right:{\small3}, fill, inner sep=.8pt, outer sep=0pt] (C) at (3,0){};
	\node[circle, label=above:{\small2}, fill, inner sep=.8pt, outer sep=0pt] (D) at (4.5,2.598075){};
	\draw [midarrow={>}] (A)--(B) node[font=\small, midway, left]{$\eta$};
	\draw [midarrow={>}] (B)--(D) node[font=\small, midway, above]{$x$};
	\draw [midarrow={>}] (D)--(C) node[font=\small, midway, below]{$\;g$};
	\draw [midarrow={>}] (A)--(C) node[font=\small, midway, below]{$\eta\cdot x\cdot g$};
	\draw [midarrow={>}] (B)--(C) node[font=\footnotesize, midway, right]{$x\cdot g$};
    	\node at (1.5,1){$\circlearrowright$};
	\node at (3,1.7){$\circlearrowright$};
	\draw [<->] (4.7,1.25)--(5.35,1.25);
	\fill[red, opacity=.7] (A) circle (3pt);
	\fill[red, opacity=.7] (B) circle (3pt);
	\draw[line width=2pt, red, opacity=.7] (A)--(B);
\end{tikzpicture}
\begin{tikzpicture}[thick] 
	\node[circle, label=left:{\small0}, fill, inner sep=.8pt, outer sep=0pt] (A) at (0,0){};
	\node[circle, label=above:{\small1}, fill, inner sep=.8pt, outer sep=0pt] (B) at (1.5,2.598075){};
	\node[circle, label=right:{\small3}, fill, inner sep=.8pt, outer sep=0pt] (C) at (3,0){};
	\node[circle, label=above:{\small2}, fill, inner sep=.8pt, outer sep=0pt] (D) at (4.5,2.598075){};
	\draw [midarrow={>}] (A)--(B) node[font=\small, midway, left]{$\eta$};
	\draw [midarrow={>}] (B)--(D) node[font=\small, midway, above]{$x$};
	\draw [midarrow={>}] (D)--(C) node[font=\small, midway, below]{$\;g$};
	\draw [midarrow={>}] (A)--(C) node[font=\small, midway, below]{$\eta\cdot x\cdot g$};	
	\draw [midarrow={>}] (A)--(D) node[font=\footnotesize, midway, above]{$\eta\cdot x$};
     \node (E) at (1.9,2.0){$\circlearrowright$};
	 \node (F) at (2.6,0.8){$\circlearrowright$};
	\fill[red, opacity=.7] (A) circle (3pt);
	\fill[red, opacity=.7] (B) circle (3pt);
	\draw[line width=2pt, red, opacity=.7] (A)--(B);
\end{tikzpicture}$$

\noindent requires the condition

\begin{equation} \label{cond3}
 \forall \eta \in H\; x \in X\; g \in G\;\; \gamma(\eta, x)\gamma^{-1}(\eta, x\cdot g)\beta(\eta\cdot x, g)\beta^{-1}(x, g)=1 
\end{equation}

\noindent to ensure invariance under the move, which condition also ensures that the product of local state values is unchanged in passing from an admissibly colored triangle to any of the admissible colorings of the 3 triangles induced by choosing an element of $G$ with which to label the one new edge not incident with the curve introduced by a 1-3 move, as, for example

$$
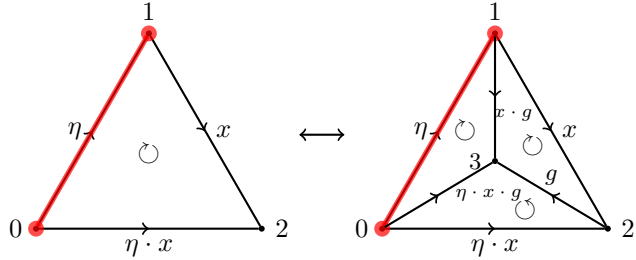
\begin{tikzpicture}[thick] 
	\node[circle, label=left:{\small0}, fill, inner sep=.8pt, outer sep=0pt] (A) at (0,0){};
	\node[circle, label=above:{\small1}, fill, inner sep=.8pt, outer sep=0pt] (B) at (1.5,2.598075){};
	\node[circle, label=right:{\small2}, fill, inner sep=.8pt, outer sep=0pt] (C) at (3,0){};
	\begin{scope}
		\draw [midarrow={>}] (A)--(B) node[font=\small, midway,left]{$\eta$};
		\draw [midarrow={>}] (B)--(C) node[font=\small, midway,right]{$x$};
		\draw [midarrow={>}] (A)--(C) node[font=\small, midway,below]{$\eta\cdot x$};
		\draw [line width=2pt, red, opacity=.7] (A)--(B);
		\fill [red, opacity=.7] (A) circle (3pt);
		\fill [red, opacity=.7] (B) circle (3pt);
		\node at (1.5,1){$\circlearrowright$};
		\draw [<->] (3.5,1.25)--(4.1,1.25);
	\end{scope}
\end{tikzpicture} 
\begin{tikzpicture}[thick] 
	\node[circle, label=left:{\small0}, fill, inner sep=.8pt, outer sep=0pt] (A) at (0,0){};
	\node[circle, label=above:{\small1}, fill, inner sep=.8pt, outer sep=0pt] (B) at (1.5,2.598075){};
	\node[circle, label=right:{\small2}, fill, inner sep=.8pt, outer sep=0pt] (C) at (3,0){};
	\node[circle, label=left: {\footnotesize3}, fill, inner sep=.8pt, outer sep=0pt] (D) at (1.5,0.9){};
	\begin{scope}
		\draw [midarrow={>}] (A)--(B) node[font=\small, midway,left]{$\eta$};
		\draw [midarrow={>}] (B)--(C) node[font=\small, midway,right]{$x$};
		\draw [midarrow={>}] (A)--(C) node[font=\small, midway,below]{$\eta\cdot x$};
		\draw [line width=2pt, red, opacity=.7] (A)--(B);
		\fill [red, opacity=.7] (A) circle (3pt);
		\fill [red, opacity=.7] (B) circle (3pt);
		\draw [midarrow={>}] (B)--(D) node[font=\tiny, midway, below]{$\;\;\;\;\;\;x\cdot g$};
		\draw [midarrow={>}] (C)--(D) node[font=\footnotesize, midway, above]{$g$};
		\draw [midarrow={>}] (A)--(D) node[font=\tiny, midway, right]{$\;\eta\cdot x\cdot g$};
		\node (E) at (1.1,1.3){$\circlearrowright$};
		\node (F) at (2,1.1){$\circlearrowright$};
		\node (G) at (1.9,0.25){$\circlearrowleft$};
	\end{scope}
\end{tikzpicture}$$

Invariance under the extended 1-2 move on the curve requires a fourth condition:

$$
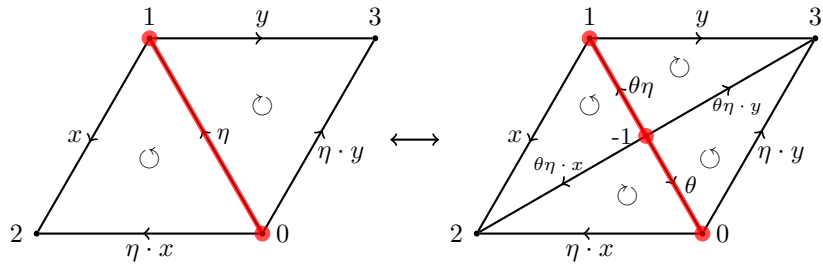
\begin{tikzpicture}[thick] 
	\node[circle, label=left:{\small2}, fill, inner sep=.8pt, outer sep=0pt] (A) at (0,0){};
	\node[circle, label=above:{\small1}, fill, inner sep=.8pt, outer sep=0pt] (B) at (1.5,2.598075){};
	\node[circle, label=right:{\small0}, fill, inner sep=.8pt, outer sep=0pt] (C) at (3,0){};
	\node[circle, label=above:{\small3}, fill, inner sep=.8pt, outer sep=0pt] (D) at (4.5,2.598075){};
	\begin{scope}
		\draw [midarrow={<}] (A)--(B) node[font=\small, midway, left]{$x$};
		\draw [midarrow={>}] (B)--(D) node[font=\small, midway, above]{$y$};
		\draw [midarrow={<}] (D)--(C) node[font=\small, midway, below]{$\;\;\;\;\;\;\eta\cdot y$};
		\draw [midarrow={<}] (A)--(C) node[font=\small, midway, below]{$\eta\cdot x$};
		\draw [midarrow={<}] (B)--(C) node[font=\footnotesize, midway, right]{$\eta$};
	    	\node at (1.5,1){$\circlearrowleft$};
		\node at (3,1.7){$\circlearrowright$};
		\draw [<->] (4.7,1.25)--(5.35,1.25);
		\fill[red, opacity=.7] (B) circle (3pt);
		\fill[red, opacity=.7] (C) circle (3pt);
		\draw[line width=2pt, red, opacity=.7] (B)--(C);
	\end{scope}
\end{tikzpicture}
\begin{tikzpicture}[thick] 
	\node[circle, label=left:{\small2}, fill, inner sep=.8pt, outer sep=0pt] (A) at (0,0){};
	\node[circle, label=above:{\small1}, fill, inner sep=.8pt, outer sep=0pt] (B) at (1.5,2.598075){};
	\node[circle, label=right:{\small0}, fill, inner sep=.8pt, outer sep=0pt] (C) at (3,0){};
	\node[circle, label=above:{\small3}, fill, inner sep=.8pt, outer sep=0pt] (D) at (4.5,2.598075){};
	\node[circle, label=left:{\footnotesize-1}, fill, inner sep=.8pt, outer sep=0pt] (E) at (2.25,1.299){};
	\begin{scope}
		\draw [midarrow={<}] (A)--(B) node[font=\small, midway, left]{$x$};
		\draw [midarrow={>}] (B)--(D) node[font=\small, midway, above]{$y$};
		\draw [midarrow={<}] (D)--(C) node[font=\small, midway, below]{$\;\;\;\;\;\;\eta\cdot y$};
		\draw [midarrow={<}] (A)--(C) node[font=\small, midway, below]{$\eta\cdot x$};
		\draw [midarrow={>}] (E)--(C) node[font=\footnotesize, midway, right]{$\theta$};
		\draw [midarrow={>}] (E)--(B) node[font=\footnotesize, midway, right]{$\theta\eta$};	
		\draw [midarrow={>}] (E)--(D) node[font=\tiny, midway, below]{$\;\;\theta\eta\cdot y$};
		\draw [midarrow={>}] (E)--(A) node[font=\tiny, midway, above]{$\theta\eta\cdot x\;$};
	     	\node at (1.5,1.7){$\circlearrowleft$};
		\node at (2,0.5){$\circlearrowright$};
		\node at (2.7,2.2){$\circlearrowright$};
		\node at (3.1,1){$\circlearrowleft$};
		\fill[red, opacity=.7] (B) circle (3pt);
		\fill[red, opacity=.7] (C) circle (3pt);
		\fill[red, opacity=.7] (E) circle (3pt);		
		\draw[line width=2pt, red, opacity=.7] (B)--(C);
	\end{scope}
\end{tikzpicture}$$
$$\gamma^{-1}(\eta, x)\gamma(\eta, y)=\gamma^{-1}(\theta\eta, x)\gamma(\theta\eta, y)\gamma^{-1}(\theta, \eta\cdot y)\gamma(\theta, \eta\cdot x)$$

Or, solving to move all the factors to one side, 

\begin{equation} \label{cond4}
\begin{split}
\forall x, y\in X \; & \eta, \theta\in H  \;\;\;
\gamma(\theta, \eta\cdot y)\gamma^{-1}(\theta\eta, y)\cdot \\
& \gamma(\eta, y)\gamma^{-1}(\theta, \eta\cdot x)\gamma(\theta\eta, x)\gamma^{-1}(\eta, x)  = 1.
\end{split}
\end{equation}

The foregoing discussion has almost established our main theorem:

\begin{theorem}
Let $G$, $H$ be groups with commuting group actions on a set $X$.  For a field $K$, let $\alpha: G\times G\to K^\times$, $\beta: X\times G\to K^\times$, and $\gamma: H\times X\to K^\times$ be functions that satisfy, $\forall f, g, h\in G,\; x, y \in X,\; \eta, \theta \in H$,
\begin{equation}
	\alpha(f,g)\alpha^{-1}(f, gh)\alpha(fg, h)\alpha^{-1}(g,h)=1 \tag{\ref{cond1}}
\end{equation}
\begin{equation}
	\beta(x, g)\beta^{-1}(x,gh)\beta(x\cdot g, h)\alpha^{-1}(g, h)=1 \tag{\ref{cond2}}
\end{equation}
\begin{equation}
	\gamma(\eta, x)\gamma^{-1}(\eta, x\cdot g)\beta(\eta\cdot x, g)\beta^{-1}(x, g)=1 \tag{\ref{cond3}}
\end{equation}
\begin{equation}
	\gamma(\theta, \eta\cdot y)\gamma^{-1}(\theta\eta, y)\gamma(\eta, y)\gamma^{-1}(\theta, \eta\cdot x)\gamma(\theta\eta, x)\gamma^{-1}(\eta, x)=1  \tag{\ref{cond4}}
\end{equation}

For any flag-like triangulation $\mathcal{T}$, let 
$$Z(\Sigma, C) := |G|^{-|\mathcal{T}_0^0|}|H|^{-|\mathcal{T}_0^1|}\sum_{\lambda \in \text{adm}(\mathcal{T})}
	\left(\prod_{\sigma\in\mathcal{T}_2^0}\alpha^{\epsilon(\sigma)}(\tilde{\lambda}(\sigma))\right)\cdot$$$$
	\left(\prod_{\sigma\in\mathcal{T}_2^1}\beta^{\epsilon(\sigma)}(\tilde{\lambda}(\sigma))\right)\cdot
	\left(\prod_{\sigma\in\mathcal{T}_2^2}\gamma^{\epsilon(\sigma)}(\tilde{\lambda}(\sigma))\right).$$

\noindent where $\text{adm}(\mathcal{T})$ denotes the set of admissible $(H,X,G)$-colorings of $\mathcal T$ and $\epsilon(\sigma)$ and $\tilde{\lambda}(\sigma)$ are as above.
Then $Z(\Sigma, C)$ is independent of $\mathcal{T}$ and therefore is a topological invariant of $\Sigma\supseteq C$.
\end{theorem}

\begin{proof}
The one issue that remains to resolve to establish the result is the independence from the ordering of the vertices.

In each of the instances of the 1-3 move above, the new vertex was added to the ordering after the vertices of the unsubdivided triangle, while in the extended 1-2 move it was added earlier in the ordering than the vertices of the subdivided edge of the curve.   In fact, inserting the new vertex elsewhere in the ordering will, in each instance give an equation which, after suitable substitutions, is an instance of one of equations (1) - (4) in the theorem.  

The unedifying calculation we, for the most part, leave to the reader.  By way of example, consider inserting the new vertex in a 1-3 move into the ordering between 0 and 1

$$
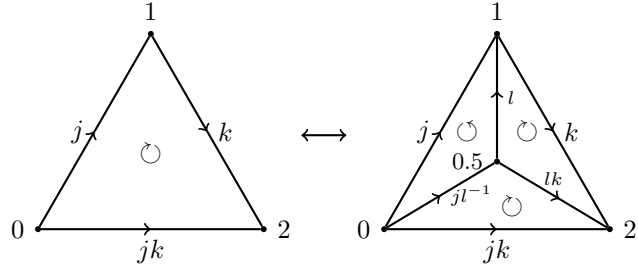
\begin{tikzpicture}[thick] 
	\node[circle, label=left:{\small0}, fill, inner sep=.8pt, outer sep=0pt] (A) at (0,0){};
	\node[circle, label=above:{\small1}, fill, inner sep=.8pt, outer sep=0pt] (B) at (1.5,2.598075){};
	\node[circle, label=right:{\small2}, fill, inner sep=.8pt, outer sep=0pt] (C) at (3,0){};
	\begin{scope}
		\draw [midarrow={>}] (A)--(B) node[font=\small, midway,left]{$j$};
		\draw [midarrow={>}] (B)--(C) node[font=\small, midway,right]{$k$};
		\draw [midarrow={>}] (A)--(C) node[font=\small, midway,below]{$jk$};
		\node (D) at (1.5,1){$\circlearrowright$};
		\draw [<->] (3.5,1.25)--(4.1,1.25);
	\end{scope}
\end{tikzpicture} 
\begin{tikzpicture}[thick] 
	\node[circle, label=left:{\small0}, fill, inner sep=.8pt, outer sep=0pt] (A) at (0,0){};
	\node[circle, label=above:{\small1}, fill, inner sep=.8pt, outer sep=0pt] (B) at (1.5,2.598075){};
	\node[circle, label=right:{\small2}, fill, inner sep=.8pt, outer sep=0pt] (C) at (3,0){};
	\node[circle, label=left: {\footnotesize0.5}, fill, inner sep=.8pt, outer sep=0pt] (D) at (1.5,0.9){};
	\begin{scope}
		\draw [midarrow={>}] (A)--(B) node[font=\small, midway,left]{$j$};
		\draw [midarrow={>}] (B)--(C) node[font=\small, midway,right]{$k$};
		\draw [midarrow={>}] (A)--(C) node[font=\small, midway,below]{$jk$};
		\draw [midarrow={<}] (B)--(D) node[font=\tiny, midway,right]{$l$};
		\draw [midarrow={<}] (C)--(D) node[font=\tiny, midway,above]{$lk$};
		\draw [midarrow={>}] (A)--(D) node[font=\tiny, midway,right]{$jl^{-1}$};
		\node (E) at (1.1,1.3){$\circlearrowleft$};
		\node (F) at (1.9,1.3){$\circlearrowright$};
		\node (G) at (1.7,0.3){$\circlearrowright$};
	\end{scope}
\end{tikzpicture}$$

\noindent giving the equation

\[ \alpha^{-1}(j,k)=\alpha(jl^{-1},l)\alpha^{-1}(l,k)\alpha^{-1}(jl^{-1},lk)\]

\noindent or equivalently

\[ \alpha^{-1}(jl^{-1},l) \alpha(jl^{-1},lk)\alpha^{-1}(j,k)\alpha(l,k)=1 \]

\noindent which last equation is an instance of  (1) with $f=jl^{-1}$, $g=l$, and $h=k$.

As was the case with the untwisted invariant, and by the same argument -- use flag-like extended Pachner moves to obtain a triangulation with no vertices in common with the original, then undo the moves, putting the vertices into the ordering any desired different order -- invariance under the extended Pachner moves, independent of where in the ordering new vertices are inserted by moves introducing new vertices, suffices to show that the quantity is independent of the ordering of vertices, subject, of course, to the condition that the vertices on the curve occur before those off the curve.
\end{proof}

\section{Examples of initial data}

It remains to see that there are non-trivial examples of the initial data needed for the construction.

\begin{example}\textbf{Restricted 2-cocycles}

Fix a group $\Gamma$ and a 2-cocycle $\hat{\alpha}:\Gamma \times \Gamma \rightarrow K^\times$ for a field $K$.  Let $H$ and $G$ be subgroups of $\Gamma$ and $X$ any subset of $\Gamma$ closed under right multiplication by elements of $G$ and left multiplication by elements of $H$, for instance $X = \Gamma$, $X = \{e\}$ or $X = HG = \{ hg | h\in H,\; g\in G\}$.

Define $\alpha=\hat{\alpha}|_{G\times G}$, $\beta=\hat{\alpha}|_{X\times G}$, and $\gamma=\hat{\alpha}|_{H\times X}$.  Since $\alpha$, $\beta$, and $\gamma$ are now just restrictions on $\hat{\alpha}$, then conditions  $(1)$, $(2)$, and $(3)$ on $\alpha$, $\beta$, and $\gamma$ reduce to instances of the 2-cocycle condition and are consequently satisfied.  Also from the 2-cocycle condition, it follows that 
\begin{align*}
	\hat{\alpha}(\theta, \eta)&=\hat{\alpha}(\theta, \eta w)\hat{\alpha}^{-1}(\theta\eta, w)
\hat{\alpha}(\eta, w)\\
	&=\gamma(\theta,\eta w)\gamma^{-1}(\theta\eta,w)\gamma(\eta,w)
\end{align*}
and
\begin{align*}
	\hat{\alpha}^{-1}(\theta, \eta)&=\hat{\alpha}^{-1}(\theta, \eta x)
\hat{\alpha}(\theta\eta, x)\hat{\alpha}^{-1}(\eta, x)\\
	&=\gamma^{-1}(\theta,\eta x)\gamma(\theta\eta,x)\gamma^{-1}(\eta,x),
\end{align*}
thus condition $(4)$ is satisfied as well.

\end{example}

Our second example can be used to construct initial data for a twisted invariant based on any triple $(H,X,G)$, and is constructed out of families of group characters.

\begin{example}

First, for any triple $(H,X,G)$, the (double) orbit of an element $x \in X$ is given by

\[ [x] := \{ hxg\; | \; h\in H, \; g \in G\} \]

Denote the set of orbits by $\tilde{X}$.  Now, for any field $K$, let the set of multiplicative characters of a group $\Gamma$ valued in $K^\times$ be denoted ${\rm Ch}(\Gamma, K)$.

Any pair of functions $\phi:\tilde{X} \rightarrow {\rm Ch}(G, K)$, $\psi:\tilde{X} \rightarrow {\rm Ch}(H, K)$, then determine a twisting $(\alpha, \beta, \gamma)$ for which $\alpha \equiv 1$ by

\[ \beta(x, g) := \phi([x])(g) \;\;\;\;\; \gamma(\eta, x) := \psi([x])(\eta) \]

Condition (1) holds trivially, while the others follow easily from homomorphic property of characters and the independence of the character applied to the group element from the representative of the orbit in $X$ occuring in the other coordinate of the input.
\end{example}

\end{document}